\documentclass[11pt,a4paper]{article}

\usepackage{amsmath,amsthm,amssymb}
\usepackage{times}
\usepackage{enumerate}
\usepackage[margin=3cm]{geometry}

\usepackage{setspace}

\usepackage{dsfont}

\newcommand{\R}{\mathbb R}

\newcommand{\gep}{\varepsilon}

\theoremstyle{definition}
\newtheorem{thm}{Theorem}[section]
\newtheorem{cor}[thm]{Corollary}
\newtheorem{prop}[thm]{Proposition}
\newtheorem{lem}[thm]{Lemma}
\newtheorem{defin}[thm]{Definition}
\newtheorem{rem}[thm]{Remark}

\newtheorem{exa}[thm]{Example}

\newcommand{\subjclass}[1]{\bigskip\noindent\emph{2010 Mathematics Subject Classification:}\enspace#1}
\newcommand{\keywords}[1]{\noindent\emph{Keywords:}\enspace#1}

\numberwithin{equation}{section}

\begin{document}

\title{A Hamiltonian 
approach to small time local attainability\\ of manifolds for nonlinear control systems}
\author{Pierpaolo Soravia\thanks{email: soravia@math.unipd.it. Orcid ID: 0000-0002-0951-4871 }\\
Dipartimento di Matematica\\ Universit\`{a} di Padova, via Trieste 63, 35121 Padova, Italy}

\date{}
\maketitle

\begin{abstract} This paper develops a new approach to small time local attainability of smooth manifolds of any dimension, possibly with boundary and to prove H\"older continuity of the minimum time function. We give explicit pointwise conditions of any order by using higher order hamiltonians which combine derivatives of the controlled vector field and the functions that locally define the target.
For the controllability of a point our sufficient conditions extend some classically known results for symmetric or control affine systems, using the Lie algebra instead, but for targets of higher dimension our approach and results are new. We find our sufficient higher order conditions explicit and easy to use for targets with curvature and general control systems .
\end{abstract}

\subjclass{Primary 49L20; Secondary 93B05, 35F21, 35D40.}
\keywords{Control theory, controllability conditions, minimum time function, Hamilton Jacobi equation, H\"older regularity}

\section{Introduction}

In this paper, we consider a nonlinear control system, $F:\R^n\times A\to\R^n$,
\begin{equation}\label{eqsystem}\left\{\begin{array}{ll}
\dot x_t=F(x_t,a_t),\\
x_0=x\in\R^n.
\end{array}\right.\end{equation}
Controllability of (\ref{eqsystem}) is a classical subject in optimal control. Given a closed set $\mathcal T\subset\R^n$, our target,
we are interested in the behaviour of trajectories, the solutions of (\ref{eqsystem}), in the neighborhood of a given point $x_o\in {\mathcal T}\backslash\hbox{int }\mathcal T$. We want to find trajectories reaching $\mathcal T$ in small time starting at any point $x$ in the neighborhood of $x_o$
in order to make the target small time locally attainable (STLA) in the neighborhood of $x_o$, or equivalently the minimum time function $T$ continuous (and vanishing) at $x_o$.
Our target is locally the solution of a system of $h(\leq n)$ smooth and independent equations and at most an inequality (if $h<n$), i.e. a manifold with boundary. 
The two extreme cases in our study are $\mathcal T=\{x_o\}$, the point, and  $\mathcal T$ being the closure of an open set, the {\it fat} target. 
Particularly the former is a classical subject of geometric control theory, and it has been deeply studied in the literature. The latter has become interesting in more recent years because it presents difficulties of a different nature and it has implications in the study of regularity of solutions to Hamilton Jacobi Bellman equations.
We will give answers also for all intermediate dimensions of the target, where the literature is much less expanded, as far as we know. We are not going to make any assumptions on the algebraic structure of the controlled vector field $F$, while in the literature one mostly finds results for symmetric, or control affine systems.
We point out that all our arguments are local in the neighborhood of the given point $x_o\in\mathcal T$, but for notational convenience, we prefer to think the system globally defined. The system will always guarantee existence, uniqueness and uniform boundedness of trajectories, see (\ref{eqlip}) below,
and control functions $a_t,t>0$ will always be piecewise constant. We will often say that $f:\R^n\to\R^n$ is an available vector field of the system, if there is $a\in A$ such that $f(x)\equiv F(x,a)$.

Our main focus is on finding simple and efficient sufficient conditions for regularity of the minimum time function $T$, rather than the perspective of geometric control, namely discussing the structure of the Lie algebra leading to local controllability properties. We will obtain the continuity of $T$ at $x_o$ by proving local estimates of the form
\begin{equation}\label{eqestt}
T(x)\leq C|x-x_o|^{1/k},\quad x_o\in\mathcal T,\;x\in B_\delta(x_o)\end{equation}
for an appropriate positive integer $k$. 
We say that the system satisfies a $k-$th order attainability condition at $x_o$ when (\ref{eqestt}) holds.
It is well known, see e.g. our paper with Bardi \cite{baso2} and the book \cite{bcd}, that if this holds true at every point of the target then the minimum time function is continuous in its open domain, while it is usually only lower semicontinuous in general, given appropriate convexity assumptions on $F$.
When the previous estimate can be improved to
\begin{equation}\label{eqestd}
T(x)\leq Cd(x,\mathcal T)^{1/k},\quad x_o\in\mathcal T,\;x\in B_\delta(x_o)\end{equation}
 at every point of the target, then it also known that $T$ becomes moreover locally $1/k-$H\"older continuous in its domain. 
Here $d(x,\mathcal T)$ indicates the distance of $x$ from the target.
We think that it is preferable to split these two steps of the quest for regularity of $T$, since proving
(\ref{eqestt}) allows us to extend some of the arguments of the case of the point target, while going from (\ref{eqestt}) to (\ref{eqestd}) is not always obvious and is not yet clarified for general systems. This second step is clear if we can
choose the constants $C,k,\delta$ in (\ref{eqestt}) uniformly in $x_o$, as it is usually done in the literature, although this is far from being necessary in general. We will not discuss it in full detail in the present paper, see our other paper \cite{so2} for some positive results. Of course when $\mathcal T$ is a point, the two estimates (\ref{eqestt}), (\ref{eqestd}) coincide. Recall that when $T$ is continuous, it is the unique solution of the Hamilton-Jacobi partial differential equation, solved by the classical Bellman approach, and H\"older regularity of $T$ also determines the regularity of solutions of more general HJ equations based on the same control system, see \cite{bfs}.

For the point target, usually one seeks answers in the Lie algebra of the available vector fields, for instance the well known full rank condition. This is sufficient for symmetric systems, as a consequence of the classical Chow-Rashevski \cite{cho} result. For affine systems, some classical results as for instance Sussmann \cite{su0,su} seek properties on the Lie algebra of the available vector fields in addition to the rank condition in order to obtain controllability at equilibrium points, see in particular the work by Frankowska \cite{fra} and Kawsky \cite{kaw}, that extend the classical Kalman condition for linear systems, see \cite{ka}. Our perspective, that also applies to general targets, is different. In order to explain it, suppose for now that locally in the neighborhood of $x_o$, $\mathcal T=\{x:u(x)=u(x_o)\}$ where $u:\R^n\to\R^h$. Our general idea is to start by finding directions $v\in\R^h\backslash\{0\}$ of $k-$th order variation of $u$ at $x_o$, namely such that there is a trajectory of (\ref{eqsystem}) with $x_0=x_o$ such that
$$u(x_t)=u(x_o)+vt^k+t^ko(1),\quad\hbox{as }t\to0+.$$
We will construct such directions quite explicitly by using higher order hamiltonian expressions that involve derivatives of the vector fields and those of the function $u$, obtained by Taylor expansions of the composed function $u(x_t)$. Next we will require the family of such directions to be a positive basis of $\R^h$, in order to prove that the target is STLA in the neighborhood of $x_o$ and (\ref{eqestt}) holds. 
Roughly speaking, and notice that $h$ may be smaller than $n$, the system can move in directions that project to a positive basis of the normal space of the target at $x_o$. We stress that we only study the family of directions at the point $x_o$.
Notice that if the target is the single point $x_o$ then we can choose $u(x)=x-x_o$. Thus we feel that the family of such directions is the correct counterpart of Lie algebra in our case and the higher order hamiltonians we consider is the counterpart of the Lie bracket. 
We build reference trajectories starting at $x_o$, one for each element of the positive basis, and we estimate their distance with the trajectory starting at a generic point $x$ in the neighborhood of $x_o$ and using the same controls one after the other. Eventually with a fixed point argument and extending an implicit function result, originally shown in \cite{pe2}, we get controllability of the target. 
Our results also apply to certain classes of nonsmooth fat targets and our second order sufficient conditions are also necessary for fat targets as shown in \cite{so2}.
When the target is not a point, then higher derivatives of $u$, such as the curvature of the target for instance, are important as well as the fact that one can reach the target at higher order with just one vector field with no need of exploiting the Lie algebra. This fact is incorporated in our sufficient conditions.
The concept of positive basis of a vector space was used in Petrov \cite{pe2} to determine that a point is first order STLA if the available vector fields of a nonlinear control system form at the point a positive basis of $\R^n$. Our result is an extension of that classical statement for sufficient conditions of any order and targets of any dimension.
Often in the literature Taylor estimates of the trajectories and of the distance function from the target are obtained separately and next combined. 
We remark on one hand that a single expansion of a scalar function (i.e. the distance function from the target) will not be enough to obtain STLA of low dimensional targets with assumptions on the data only at $x_o$, and on the other that
the derivative of the composition $u(x_t)$ can be easily expressed in terms of higher order hamiltonians. 
In the case of the point target, we find in our proof either the classical Chow-Rashevski \cite{cho} controllability of symmetric systems and the one for affine systems of Frankowska \cite{fra} and Kawsky \cite{kaw} but our main focus is for higher dimensions of the target.
We also notice that usually in the literature, the proof that a target is STLA follows different paths in the case of the point and of fat targets.
Here we give a unified presentation for the two cases as well as for all intermediate dimensions of the target.

Our approach initiated in \cite{so1} where we studied symmetric systems and \cite{so2} where we considered general nonlinear systems. In both papers we consider a fat target and find simple algebraic second order sufficient conditions to show that the target is STLA with trajectories of at most one switch and a degenerate elliptic differential inequality to express them.
We expect to be able to use a similar approach to prove small time local capturability of pursuit evasion differential games, where higher order sufficient conditions are completely missing from the literature, as far as we know. Also control problems with state constraints can be naturally approached with our methods. We will undertake such problems in the future.

Most attainability results in the literature concern controllability to a point. 
Besides what we already mentioned,
Liverowskii \cite{li} extended the approach by Petrov, see also \cite{pe}, to prove second order sufficient and necessary conditions for the point. 
For other results on higher order necessary and sufficient conditions for affine systems, we refer to Bianchini and Stefani \cite{bs,bs2,bs3,ste} and to Krastanov \cite{kr2} for dynamics on manifords. 
We also mention the chapter on controllability of control systems in the book by Coron \cite{cor} where many additional references can be found. For fat targets,
Bardi-Falcone \cite{BaFa90} found necessary and sufficient first order conditions, while more recently our paper with Bardi and Feleqi \cite{bfs} derives necessary second order conditions and drops one level of regularity for sufficient conditions, by using the generalized Lie brackets of Rampazzo-Sussman \cite{rasu}. 
For affine systems with drift vanishing on the target, Krastanov and Quincampoix \cite{kr,kr2,kr3} proved higher order sufficient STLA conditions of a different nature for nonsmooth fat targets only involving the normal vectors and using the idea of variations of the reachable set.
The work by Marigonda and Rigo \cite{ma,ma2}
pointed out the importance of the geometry of the target, in particular its curvature, and studied higher order attainability of certain nonsmooth targets for affine systems with nontrivial drift and local H\"older continuity of the minimum time function. Later with Le \cite{ma3} they studied higher order sufficient conditions focusing on the presence of state constraints.
For attainability of a smooth target with intermediate dimension much fewer results are available in the classical literature. We know of the papers by Bacciotti \cite{ba} for first order conditions and the author \cite{so} for second order conditions for symmetric systems and smooth targets of any dimension, possibly with a boundary. 
We finally mention Motta and Rampazzo \cite{mr} who construct higher order hamiltonians adding iterated Lie brackets as additional vector fields to prove global asymptotic controllability. Their Hamiltonian is still a first order operator in contrast to ours. Recently Albano, Cannarsa and Scarinci \cite{alcasc1} show for some symmetric systems that the set where the local Lipschitz continuity of the minimum time function fails is the union of singular trajectories, and that it is analytic except on a null set. For results in this direction see also the author in \cite{so3}.

We outline the contents of the paper. In Section 2 we develop the basic asymptotic formulas for trajectories and some calculus in order to understand relationships with previous literature and build the tools to check the examples. Sections 3--4 are the core of the paper and present the sufficient conditions for STLA. We work out three cases separately: the case of the fat target because it is easier and does not use the Lemma in the Appendix; the case of the point because it does not use an extra assumption and develops some preliminary tool needed in the general proof. Finally the general case of intermediate dimensions of the target is dealt with. Section 5 shows some examples where we apply our results. The Appendix contains a revisited Petrov's Lemma which is a key step in the arguments.

\section{Preliminaries and notations: Hamiltonian asymptotic formulas}

We initiate with some basic facts on Taylor estimates to establish the decrease rate of scalar functions along trajectories of dynamical systems. This will illustrate what we mean by a Hamiltonian approach. Next we will proceed with trajectories of control systems.
What we do has connections with classical formalisms to compute the asymptotics of flows of dynamical systems, such as for instance chronological calculus by Agrachev and Gamkrelidze, see e.g. \cite{ag2,ag1,ag3}, or products of exponentials by Sussmann, see e.g. \cite{su,ks}, but in that work we did not find what we specifically need here.
We start with a function $u:\R^n\to\R$ and consider the trajectories of a dynamical system
\begin{equation}
\left\{\begin{array}{ll}\dot x_t=f(x_t),\\
x_0\in {\mathbb R}^n,
\end{array}\right.
\end{equation}
where $f\in C(\R^n;\R^n)$ is a vector field.
For integer $k\geq1$, $f\in C^{k-1}(\R^n;\R^n)$, we introduce the Hamiltonian operators $H^{(h)}_f:C^h(\R^n)\to C(\R^n)$, $h=0,\dots,k$,
$$H^{(0)}_fu\equiv u,\quad H_fu=f\cdot\nabla u,\quad H^{(h+1)}_fu=H_f\circ H^{(h)}_fu.$$
Note that $H_fu$ in the literature also appears as Lie derivative of $u$ (or pre-hamiltonian). Observe that, in any interval where the trajectory is defined,
$$\frac d{dt}u(x_t)=H_fu(x_t),\quad \frac {d^k}{dt^k}u(x_t)=H^{(k)}_fu(x_t)$$
and therefore the following property easily follows. 
\begin{lem} If $f\in C^{k-1}(\R^n;\R^n)$ is Lipschitz continuous and $u\in C^k(\R^n)$, then the following Taylor formula holds
\begin{equation}\label{eqtaylorfirst}
u(x_t)=\sum_{i=0}^{k}\frac{t^i}{i!}H^{(i)}_fu(x_0)+t^ko(1),\quad\hbox{as }t\to0.
\end{equation}
\end{lem}
\begin{rem} The remainder term $o(1)$ in (\ref{eqtaylorfirst}) can be expressed as 
$$\frac{1}{k!}\left(H^{(k)}_fu(x_s)-H^{(k)}_fu(x_0)\right),$$
for a suitable $s\in(0,t)$ and then it goes to 0 as $t\to0$ locally uniformly for $x_0\in\R^n$.
\end{rem}
We notice that the operator $H^{(2)}_f:C^2(\R^n)\to C(\R^n)$ can be explicitely written as
$$H^{(2)}_fu=H_f\circ H_fu=f\cdot\nabla(f\cdot\nabla u)=\hbox{Tr}(D^2u\;f\otimes f)+Df\;f\cdot\nabla u$$
and it is degenerate elliptic on $u$. Likewise ${\mathcal H}_{f,k}:=H^{(k)}_f$ is a partial differential operator of order $k$.

When $f,u$ are $C^\infty$ and $u(x_t)$ is analytic, we obtain, for small $t$,
$$u(x_t)=\sum_{k=0}^{+\infty}\frac{t^k}{k!}H^{(k)}_fu(x_0)=:e^{tH_f}u(x_0),$$
introducing an exponential notation for the Hamiltonian.
We now want to apply the same approach to families of vector fields and consider what we name one-switch (balanced) trajectories.
Let $t>0$ and $f,g\in C(R^n;\R^n)$ be two vector fields. 
Consider a Caratheodory solution of
\begin{equation}\label{eqswitch}
\dot x_s=\left\{\begin{array}{ll}f(x_s),\quad&\mbox{ if }s\in[0,t),\\
g(x_s),&\mbox{ if }s\in[t,2t],\end{array}\right.
\quad x_0\in\R^n.
\end{equation}
Note that $x_s[t]$ is indeed a family of trajectories indexed with the parameter $t>0$, although we will usually hide the parameter $t$.
We want to describe the variation at the end point $u(x_{2t})-u(x_0)$. 
\begin{lem}\label{lem23bb} Let $f,g\in C^{k-1}(\R^n;\R^n)$, $f,g$ Lipschitz, and $u\in C^k(\R^n)$. Then the one switch trajectory (\ref{eqswitch}) satisfies
the following asymptotic formula
\begin{equation}\label{eqtaylorsecond}
u(x_{2t})=u(x_0)+\sum_{i=1}^{k}\frac{t^i}{i!}(H_f\boxplus H_g)^iu(x_0)+t^ko(1),\quad\hbox{as }t\to0,
\end{equation}
where the remainder tends to 0 locally uniformly with respect to $x_0$. Here we define
$$(H_f\boxplus H_g)^iu:= \sum_{i=0}^{m}\left(\begin{array}{c}m\\i\end{array}\right)
H^{(m-i)}_f\circ H^{(i)}_gu.$$
\end{lem}
\begin{proof}
In the assumptions, by (\ref{eqtaylorfirst}) we know that, in the $[t,2t]$ interval,
$$u(x_{2t})=\sum_{i=0}^{k}\frac{t^i}{i!}H^{(i)}_gu(x_t)+t^ko(1),$$
as $t\to0$, where the remainder goes to 0 uniformly with respect to the initial point, which is $x_t$ in this case. 
For $i=1,\dots,k$, let $v=H^{(i)}_gu$, likewise
$$H^{(i)}_gu(x_t)=v(x_t)=\sum_{j=0}^{k-i}\frac{t^j}{j!}H^{(j)}_fv(x_0)+t^{k-i}o(1)
=\sum_{j=0}^{k-i}\frac{t^j}{j!}H^{(j)}_f\circ H^{(i)}_gu(x_0)+t^{k-i}o(1),$$
thus finally
$$\begin{array}{ll}
u(x_{2t})&=\sum_{i=0}^{k}\frac{t^i}{i!}\left(\sum_{j=0}^{k-i}\frac{t^j}{j!}H^{(j)}_f\circ H^{(i)}_gu(x_0)+t^{k-i}o(1)\right)+t^ko(1)\\
&=\sum_{i=0}^{k}\left(\sum_{j=0}^{k-i}\frac{t^{i+j}}{i!j!}H^{(j)}_f\circ H^{(i)}_gu(x_0)\right)+t^ko(1)
\\&
=\sum_{m=0}^{k}\frac{t^{m}}{m!}\sum_{i=0}^{m}\left(\begin{array}{c}m\\i\end{array}\right)
H^{(m-i)}_f\circ H^{(i)}_gu(x_0)+t^ko(1).
\end{array}$$
\end{proof}

Notice that $H_f\boxplus H_g=H_f+H_g=H_{f+g}$
is itself a Hamiltonian operator. 
However we caution the reader that, if $m\geq2$,
$$(H_f\boxplus H_g)^m\neq H_{f+g}^{(m)},$$
in general.
Indeed if $[f,g]=Dgf-Dfg$ represents the Lie bracket of two vector fields, then we immediately compute
$$H_{[f,g]}u=[f,g]\cdot \nabla u=H_f\circ H_gu-H_g\circ H_fu=\frac12((H_f\boxplus H_g)^2u-(H_g\boxplus H_f)^2u)$$
and therefore, as an example,
\begin{equation}\label{eqbasich}
(H_f\boxplus H_g)^2u=H^{(2)}_fu+2H_f\circ H_gu+H^{(2)}_gu=H^{(2)}_{f+g}u+[f,g]\cdot\nabla u.
\end{equation}
Therefore Lie brackets are part of the game, as they should be. 
Also observe that for $\lambda>0$ by definition we have that
$$(H_{\lambda f}\boxplus H_{\lambda g})^mu=\lambda^m(H_f\boxplus H_g)^mu,$$
so we have a homogeneity property. We notice the following sometimes useful algebraic properties
$$H_{-f}u=-H_fu;\;(H_{f}\boxplus H_{f})^ku=2^kH^{(k)}_fu,\;(H_{-f}\boxplus H_{-g})^ku=(-1)^k(H_{f}\boxplus H_{g})^ku.$$

\begin{rem}
It will be important to discuss the sign of the quantities $(H_f\boxplus H_g)^iu(x_0)$, choosing $f,g$ appropriately. In the special case $k=2$, from (\ref{eqbasich}) we get that if $(H_f\boxplus H_g)^2u(x_0)<0$ then either $H^{(2)}_{f+g}u(x_0)<0$ or $[f,g]\cdot \nabla u(x_0)\neq0$. We can check that the implication can be somewhat reversed. 
Let
$$S(x)=\left(\begin{array}{cc}
H_f^{(2)}u(x)\quad&H_f\circ H_gu(x)\\
H_g\circ H_fu(x)&H^{(2)}_gu(x)
\end{array}\right)$$
and observe that $S(x)$ is symmetric if and only if $[f,g]\cdot \nabla u(x)=0$. 
Indeed in \cite{so1} we proved the following.\\
{\bf Proposition} If $[f,g]\cdot \nabla u(x_0)\neq0$ then there is an eigenvector $a_1=\;^t(a_{1,1},a_{2,1})$ of $\;^tS(x_0)S(x_0)$ with strictly positive eigenvalue $\lambda^2$ ($\lambda>0$), such that if $a_2=\;^t(a_{1,2},a_{2,2})=-S(x_0)a_1/\lambda$, then $a_1\neq -a_2$ and
$$(H_{a_{1,2}f+a_{2,2}g}\boxplus H_{a_{1,1}f+a_{2,1}g})^2u(x_0)<0.
$$
\end{rem}

The contents of Lemma \ref{lem23bb} can be clearly extended to (balanced) switch trajectories with any finite number of switches.
For instance, given three vector fields $f,g,h\in C^{m-1}(\R^n;\R^n)$ and $u\in C^m(\R^n)$, we can also introduce the operators ${\mathcal H}_{f,g,h,m}:C^m(\R^n)\to C(\R^n)$, where
\begin{equation}\label{eqthreeham}
{\mathcal H}_{f,g,h,m}u(x_0)\equiv (H_f\boxplus H_g\boxplus H_h)^mu:=
\sum_{i,j,\;i+j\leq m}
\left(\begin{array}{c}m\\i,j\end{array}\right)H^{(m-i-j)}_f\circ H^{(j)}_g\circ H^{(i)}_hu.
\end{equation}
We recall that the trinomial coefficients are defined as
$$\left(\begin{array}{c}m\\i,j\end{array}\right)=\frac{m!}{(m-i-j)!i!j!}=
\left(\begin{array}{c}m\\i\end{array}\right)\left(\begin{array}{c}m-i\\j\end{array}\right).
$$
To clarify the above definition, we notice the following property.
First define the following notation
$$((H_{f}\boxplus H_{g})\boxplus H_{{h}})^ku:=
\sum_{i=0}^k\left(\begin{array}{c}k\\i\end{array}\right)\left(H_f\boxplus H_g\right)^{k-i}\circ H^{(i)}_hu.$$
\begin{prop}
Let $f,g,h\in C^{k-1}(\R^n;\R^n)$ and $u\in C^k(\R^n)$. Then
$$(H_f\boxplus H_g\boxplus H_h)^ku=((H_f\boxplus H_g)\boxplus H_h)^ku.
$$
More in general, if $f_1,\dots,f_{m+1}\in C^{k-1}(\R^n;\R^n)$, then
\begin{equation}\label{eqpowersum}
(H_{f_1}\boxplus \dots\boxplus H_{f_{m+1}})^ku=((H_{f_1}\boxplus \dots\boxplus H_{f_m})\boxplus H_{f_{m+1}})^ku,
\end{equation}
where notations are extended in a straightforward way.
\end{prop}
\begin{proof}
It is a matter of computing things. For three vector fields,
$$\begin{array}{ll}
(H_f\boxplus H_g\boxplus H_h)^k&=\sum_{i,j,\;i+j\leq k}
\left(\begin{array}{c}k\\i,j\end{array}\right)H^{(k-i-j)}_f\circ H^{(j)}_g\circ H^{(i)}_hu\\
&=\sum_{i=0}^k\left(\begin{array}{c}k\\i\end{array}\right)\left(\sum_{j=0}^{k-i}
\left(\begin{array}{c}k-i\\j\end{array}\right)H^{(k-i-j)}_f\circ H^{(j)}_g\right)\circ H^{(i)}_hu
\\&
=\sum_{i=0}^k\left(\begin{array}{c}k\\i\end{array}\right)\left(H_f\boxplus H_g\right)^{k-i}\circ H^{(i)}_hu.
\end{array}$$
The general case can be proved similarly by induction.
\end{proof}
By the previous statement, the $k-$th power of the sum of $m+1$ Hamiltonians 
can either be defined recursively by the right hand side of (\ref{eqpowersum}) or explicitly in the corresponding way to (\ref{eqthreeham}), by using the multinomial coefficients.
The previous operators that we have defined appear in the Taylor estimates of trajectories in the following way. Given vector fields $f_1,\dots,f_m\in C(\R^n;\R^n)$, a (balanced) trajectory with $(m-1)$-switches is a Caratheodory solution of
\begin{equation}\label{eqtwoswitch}
\dot x_s=\left\{\begin{array}{ll}f_1(x_s),\quad&\mbox{ if }s\in[0,t),\\
f_2(x_s),&\mbox{ if }s\in[t,2t),\\
\dots\;,\\
f_m(x_s),&s\in[(m-1)t,mt],
\end{array}\right.
\quad x_0\in\R^n.
\end{equation}
It is now clear by construction that by using an elementary induction argument on the parameters, we obtain the following result.
\begin{prop}\label{proptaylor} Let $f_1,\dots,f_m\in C^{k-1}(\R^n;\R^n)$ and $u\in C^k(\R^n)$. Then the (balanced) trajectory with $(m-1)$-switches (\ref{eqtwoswitch}) satisfies
the following asymptotic formula at the end point
\begin{equation}\label{eqtaylormult}
u(x_{mt})=u(x_0)+\sum_{i=1}^{k}\frac{t^i}{i!}(H_{f_1}\boxplus\dots\boxplus H_{f_m})^iu(x_0)+t^ko(1),\quad\hbox{as }t\to0,
\end{equation}
where the remainder tends to 0 locally uniformly in $x_0$. 
\end{prop}

\subsection{Nonlinear control systems}
We now consider the nonlinear control system (\ref{eqsystem}),
where $x_.$ is the state and $a_.$ is the control, determined by a controlled vector field $F$.
The general assumptions we make here will stand for the rest of the paper.
We assume for convenience that $A$ is a compact subset of a metric space, $F:\R^n\times A\to\R^n$ is continuous and locally Lipschitz continuous in the variable $x$ in a neighborhood of $x_o\in\R^n$, that is we can find $R, L>0$ such that
\begin{equation}\label{eqlip}
|F(x,a)-F(y,a)|\leq L|x-y|,\end{equation}
for all $x,y\in B_R(x_o)$, $a\in A$. In particular there is $M>0$ such that
$$|F(x,a)|\leq|F(x,a)-F(x_o,a)|+|F(x_o,a)|\leq L|x-x_o|+|F(x_o,a)|\leq M,$$
for all $x\in B_R(x_o)$, $a\in A$. Therefore if $(x_s)_{s\in[0,t]}$ is a trajectory solution of (\ref{eqsystem}) then
$$|x_s-x|\leq Ms,\quad s\in[0,t],$$
and $|x_s-x_o|\leq R$ if $|x-x_o|\leq R/2$ and $t\leq R/(2M)=:\sigma$.

To the control system we associate a target $\mathcal T$, a closed subset of $\R^n$. The target will be assumed smooth, in the sense that given $x_o\in\mathcal T\backslash \hbox{int}{\mathcal T}$, it is locally defined in $B_R(x_o)$ by a family of $h(\leq n)$ equations (a $h-$dimensional manifold)
$$\left\{\begin{array}{l}
u_1(x)=0,\\
\dots,\\
u_h(x)=0,
\end{array}\right.$$
possibly, when $h\leq n-1$, with an additional inequality (a $h-$dimensional manifold with boundary)
$$u_{h+1}(x)\leq0,$$
where $u_1,\dots,u_{h+1}\in C^1(\R^n)$ at least, are given functions. We always assume that the jacobian of either one of the functions involved in the definition of $\mathcal T$, $u=(u_1,\dots,u_h)$ or $\hat u=(u_1,\dots,u_{h+1})$ has full rank at $x_o$. 

We are interested in the property that $\mathcal T$ is small time local attainable (STLA for short) in the neighborhood of the given point $x_o$, namely the continuity at $x_o$ of the minimum time function
$$T(x)=\inf_{a_.\in L^\infty((0,+\infty);A)}t_x(a)(\leq+\infty),$$
where $t_x(a)=\min\{t\geq0:x_t\in\mathcal T\}(\leq+\infty)$, $x_t$ trajectory of the control system corresponding to the control $a_\cdot$ and initial point $x$. Note that $T(x_o)=0$. It is well known that the continuity of $T$ at all boundary points of the target propagates in the whole of the domain of $T$ (the reachable set) which is then an open set, see e.g. Bardi and the author \cite{baso2} or \cite{bfs}.
We will seek continuity of $T$ at $x_o$ by proving estimates of the form (\ref{eqestt})
in the neighborhood of $x_o$, for a suitable integer $k$. 

We give some definitions to classify the structure of system (\ref{eqsystem}).
\begin{defin}
We say that the system is convex if for any pair of available vector fields $f,g\in C(\R^n;\R^n)$ any convex combination
$\lambda f+(1-\lambda)g$, for all $\lambda\in[0,1]$
is also available. 

We say that the system is symmetric if it is convex and for any available vector field $f$, then also $-f$ is available.
 

We say that a system is affine (in the control) if it has the structure
$$F(x,a)=f_o(x)+G(x,a),$$
where $G(x,a)$ is symmetric. Usually $f_o$ is called the drift.

We say that a function $u:\R^n\to\R$ has a $k-th$ order decrease rate for the control system at $x_o\in\R^n$ if there are available vector fields $f_1,\dots f_m$ such that
$$(H_{f_1}\boxplus\dots\boxplus H_{f_m})^ku(x_o)<0,\;(H_{f_1}\boxplus\dots\boxplus H_{f_m})^iu(x_o)=0\;\hbox{for all }i=1,\dots,k-1.$$
\end{defin}

Notice that if $A\subset\R^m$ is convex and $F(x,a)=f_o(x)+\sum_{i=1}^ma_if_i(x)$, then the system is convex. If $A$ is moreover symmetric with respect to the origin, then it is affine in the control and symmetric if $f_o\equiv0$.

Observe that if we have as available vector fields $f,-f,g,-g$, then 
we can compute $(H_f\boxplus H_g\boxplus H_{-f}\boxplus H_{-g})u(x)=0$ and
\begin{equation}\label{eqbracket}\begin{array}{l}
(H_f\boxplus H_g\boxplus H_{-f}\boxplus H_{-g})^2u(x)=((H_f\boxplus H_g)\boxplus (H_{-f}\boxplus H_{-g}))^2u(x)\\
=(H_f\boxplus H_g)^2u(x)+(H_{-f}\boxplus H_{-g})^2u(x)+2(H_f\boxplus H_g)\circ (H_{-f}\boxplus H_{-g})u(x)
\\=2(H_f\boxplus H_g)^2u(x)-2H^{(2)}_{f+g}u(x)=
2[f,g]\cdot\nabla u(x).
\end{array}\end{equation}
Therefore if $[f,g]\cdot\nabla u(x_o)<0$ then $u$ has a second order decrease rate at $x_o$.
For general systems however we cannot produce a trajectory having the second order Taylor coefficient proportional to the Lie bracket of any two given available vector fields.

\begin{rem}
The equivalence (\ref{eqbasich}) 
gives rise to the following fully degenerate second order Hamilton-Jacobi operator taking into account pairs of available vector fields
$$\max_{(a_1,a_2)\in A\times A}\begin{array}{ll}\left\{\right.
&-\hbox{Tr}(D^2u\;(F(x,a_1)+F(x,a_2))\otimes (F(x,a_1)+F(x,a_2)))\\
&-\left.\left(D(F(x,a_1)+F(x,a_2))\;(F(x,a_1)+F(x,a_2))+[(F(x,a_1),F(x,a_2))]\right)\cdot\nabla u.
\right\}\end{array}$$
that we introduced in \cite{so1,so2} as a counterpart of the classical Bellman operator, in order to study the second order attainability of fat targets.
If $F(x_o,a)\cdot\nabla u(x_o)=0$ for all $a\in A$, then such operator applied to $u$ is strictly positive at $x_o$ if and only if $u$ has 2nd order decrease rate at $x_o$. If moreover $u\equiv d$ is the signed distance function from $\partial\mathcal T$ which is negative in the interior of the target, then $n(x_o)=\nabla u(x_o)$ is the exterior normal vector and $d$ has second order decrease rate if and only if there are $a_1,a_2\in A$ such that
$$\begin{array}{ll}
\hbox{Tr}(D^2d\;(F(x_o,a_1)+F(x_o,a_2))\otimes (F(x_o,a_1)+F(x_o,a_2)))\\
+\left(D(F(x_o,a_1)+F(x_o,a_2))\;(F(x_o,a_1)+F(x_o,a_2))+[(F(\cdot,a_1),F(\cdot,a_2))](x_o)\right)\cdot n(x_o)<0.
\end{array}$$
The first line above is proportional to the normal curvature in the direction of the average of the vector fields and the second checks the directions of an appropriate vector field and the exterior normal. It has two contributions: the first relative to the average again and the second to their Lie bracket. If in particular $a_1=a_2$ we find that $d$ has 2nd order decrease rate with only one vector field as
$$\hbox{Tr}(D^2d\;F(x_o,a_1)\otimes F(x_o,a_1))\\
+\left(DF(x_o,a_1)\;F(x_o,a_1)\right)\cdot n(x_o)<0.
$$
\end{rem}

We now develop some calculus for the Hamiltonian-Taylor coefficients to show that they can be manipulated as easily as the Lie brackets.
\begin{prop}\label{proprecursive}
Let $f,g\in C^{k}(\R^n;\R^n)$ and $u\in C^{k+1}(\R^n)$. Then for $k=1$, $(H_f\boxplus H_g)u=H_{f+g}u$, and for all $k\geq1$
\begin{equation}\label{eqrecursive}
(H_f\boxplus H_g)^{k+1}u=H_f\circ (H_f\boxplus H_g)^{k}u+(H_f\boxplus H_g)^{k}\circ H_gu.
\end{equation}
\end{prop}
\begin{proof}
We compute from the definition
$$\begin{array}{c}\sum_{i=0}^{k+1}
\left(\begin{array}{c}k+1\\i\end{array}\right)H^{(k+1-i)}_f\circ H^{(i)}_gu=H^{(k+1)}_fu+H^{(k+1)}_gu+
\sum_{i=1}^{k}\left(\begin{array}{c}k+1\\i\end{array}\right)H^{(k+1-i)}_f\circ H^{(i)}_gu\\
=H^{(k+1)}_fu+H^{(k+1)}_gu+
\sum_{i=1}^{k}\left(\begin{array}{c}k\\i\end{array}\right)H^{(k+1-i)}_f\circ H^{(i)}_gu
+\sum_{i=1}^{k}\left(\begin{array}{c}k\\i-1\end{array}\right)H^{(k+1-i)}_f\circ H^{(i)}_gu\\
=H_f\circ \sum_{i=0}^{k}\left(\begin{array}{c}k\\i\end{array}\right)H^{(k-i)}_f\circ H^{(i)}_gu+
\sum_{i=1}^{k+1}\left(\begin{array}{c}k\\i-1\end{array}\right)H^{(k+1-i)}_f\circ H^{(i)}_gu\\
=H_f\circ (H_f\boxplus H_g)^{k}u+\left(\sum_{i=0}^{k}\left(\begin{array}{c}k\\i\end{array}\right)H^{(k-i)}_f\circ H^{(i)}_g\right)\circ H_gu.
\end{array}$$
\end{proof}

%
Sometimes our Taylor operators simplify to first order. This is important, because it allows to drop the regularity of the target in our statements. The first example is the following.
\begin{rem}\label{remequilibrium}
Suppose that we have two vector fields balanced at $x_0$, i.e. $f(x_0)+g(x_0)=0$. Therefore $H_{f+g}u(x_o)=(f+g)\cdot \nabla u(x_0)=0$ and by Proposition \ref{proprecursive} we also have
$$\begin{array}{ll}
(H_f\boxplus H_g)^2u(x_0)&=H_{[f,g]}u(x_0),\\
(H_f\boxplus H_g)^3u(x_0)&=H_f\circ H_{[f,g]}u(x_0)+H_{[f,g]}\circ H_gu(x_0)=H_{[f,g]}\circ H_gu(x_0)-H_g\circ H_{[f,g]}u(x_0)\\
&=H_{[[f,g],g]}u(x_0)=\hbox{ad}^2_gf(x_0),
\end{array}$$
and then easily by induction
$$
(H_f\boxplus H_g)^{k+1}u(x_0)=(-1)^kH_{\hbox{ad}^k_gf}u(x_0),
$$
where we defined the iterated Lie bracket $\hbox{ad}_gf:=[g,f]$, $\hbox{ad}^{k+1}_gf:=[g,\hbox{ad}^k_gf]$.
Thus in this case the higher order decrease rate condition is reduced to discussing the sign of a first order operator provided by an iterated Lie bracket.
\end{rem}

A higher order Hamiltonian-Taylor operator clearly reduces to first order also when the function $u$ is linear. Let for instance I(x)=x and suppose that $u(x)=I_i(x)$ for some $i=1,\dots,n$. Notice in fact that if $f=(f_1,\dots,f_n)$ is an available vector field, then 
$$H_fI_i(x)=f\cdot \nabla I_i(x)=f_i(x).$$
If we have two vector fields $f,g$, we can then prove the following.
\begin{prop}\label{proplinearconstraint}
Let $f,g$ be two available vector fields with the appropriate regularity required in the operations of the following formulas.
Then for all $k\geq2$ and $i=1,\dots,n$ we have that
$$\begin{array}{c}
(H_f\boxplus H_g)^2I_i(x)=F_2\cdot\nabla I_i(x)=(F_2)_i(x),\\
(H_f\boxplus H_g)^{k+1}I_i(x)=F_{k+1}\cdot\nabla I_i(x)=(F_{k+1})_i(x),
\end{array}$$
where we define recursively $F_2(x)=D(f+g)(f+g)(x)+[f,g](x)$ and $F_{k+1}(x):=DF_k(f+g)(x)+[F_k,g](x)$.
Therefore $(H_f\boxplus H_g)^kI_i$ is a first order operator on $I_i$.
\end{prop}
\begin{proof}
For the two vector fields $f,g$ we obtain, if $e_i$ is the i-th element of the standard unit basis of $\R^n$, i.e. $x_i=x\cdot e_i$, since $\nabla I_i=e_i$,
$$H_f\circ H_gI_i(x)=f\cdot \nabla g_i(x)=Dgf\cdot e_i=Dgf\cdot\nabla I_i(x).$$
Thus
$$(H_f\boxplus H_g)^2I_i(x)=(D(f+g)(f+g)+[f,g])\cdot\nabla I_i(x)=F_2\cdot\nabla I_i(x).
$$
Next
$$\begin{array}{ll}
(H_f\boxplus H_g)^3I_i(x)&=H_f\circ (H_f\boxplus H_g)^2I_i(x)+(H_f\boxplus H_g)^2\circ H_gI_i(x)\\
&= f\cdot \nabla (F_2)_i(x)+F_2\cdot \nabla g_i(x)
=\{DF_2(f+g)+DgF_2-DF_2g\}\cdot\nabla I_i(x)\\
&=\left(DF_2(f+g)+[F_2,g]\right)\cdot \nabla I_i(x)=F_3\cdot \nabla I_i(x).
\end{array}$$
At this point we find similarly by induction that
$$(H_f\boxplus H_g)^{k+1}I_i(x)=\left(DF_k(f+g)+[F_k,g]\right)\cdot \nabla I_i(x)=F_{k+1}\cdot\nabla I_i(x).
$$
\end{proof}

\subsection{Vector valued H-operators for functions}

When $f_1,\dots.f_m\in C^{j-1}(\R^n;\R^n)$ are vector fields and $u:\R^n\to\R^h$, $u=(u_1,\dots,u_h)\in C^j(\R^n;\R^h)$, we will use the following notation for a vector valued operator
$$(H_{f_1}\boxplus \dots\boxplus H_{f_m})^ju:=\;^t((H_{f_1}\boxplus \dots\boxplus H_{f_m})^ju_r)_{r=1,\dots,h}.$$
Our expansion formulas therefore become statements also for vector valued functions and in particular trajectories themselves.

\begin{prop}\label{proptaylorvector} Let $f_1,\dots,f_m\in C^{k-1}(\R^n;\R^n)$ and $u\in C^k(\R^n;\R^h)$. Then the (balanced) trajectory with $(m-1)$-switches (\ref{eqtwoswitch}) satisfies
the following asymptotic formula in $\R^h$
\begin{equation}\label{eqtaylormultvector}
u(x_{mt})=u(x_0)+\sum_{i=1}^{k}\frac{t^i}{i!}(H_{f_1}\boxplus\dots\boxplus H_{f_m})^iu(x_0)+t^ko(1),\quad\hbox{as }t\to0,
\end{equation}
where the remainder tends to 0 locally uniformly in $x_0$. 
If in particular $h=n$ and $u(x)=I(x)=x$ is the identity function on $\R^n$, then we obtain the following expansion formula for balanced trajectories
\begin{equation}\label{eqtaylortraj}
x_{mt}=x_0+\sum_{i=1}^{k}\frac{t^i}{i!}(H_{f_1}\boxplus\dots\boxplus H_{f_m})^iI(x_0)+t^ko(1),\quad\hbox{as }t\to0.
\end{equation}
\end{prop}
\begin{rem}\label{remneed}
Consider two points $x_1,x_2\in\R^n$ and a family of available vector fields $f_1,\dots,f_m$. For $0<s\leq t\leq T$ consider the trajectory $x^1_r$ defined in $[0,ms]$, starting at $x_1$ and using the vector  fields $f_i$ on subsequent intervals of length $s$ as in (\ref{eqtwoswitch}) and the trajectory $x^2_r$ defined in $[0,mt]$, starting at $x_2$ and using the vector  fields $f_i$ on subsequent intervals of length $t$. Notice that by Gronwall inequality
$$\begin{array}{c}
|x^1_s-x^2_t|\leq M|t-s|+|x_1-x_2|e^{Ls},\quad
|x^1_{2s}-x^2_{2t}|\leq M|t-s|+|x^1_s-x^2_t|e^{Ls},
\end{array}$$
and then by induction $|x^1_{ms}-x^2_{mt}|\leq C(|t-s|+|x_1-x_2|)$, $C$ depending on $m,T$. The end point of a balanced trajectory of a given family of vector fields is then continuous with respect to the initial point and time. Therefore the remainder $o(1)$ in (\ref{eqtaylormultvector}) can be expressed as a continuous function in $\gamma=\gamma(x_0,t)$ given by $\gamma(x_0,0)=0$ and
$$\gamma(x_0,t)=\frac{k!}{t^k}\left(u(x_{mt})-u(x_0)-\sum_{i=1}^{k}\frac{t^i}{i!}(H_{f_1}\boxplus\dots\boxplus H_{f_m})^iu(x_0)\right),\quad\hbox{if }t>0.$$
\end{rem}
\begin{rem}
To better understand our expansion formula in the vector case, note that if $f,g$ are available vector fields then
$$\begin{array}{ll}
(H_f\boxplus H_g)^3I=H_f\circ (H^{(2)}_{f+g}I+H_{[f,g]}I)+(H^{(2)}_{f+g}+H_{[f,g]})\circ H_gI=
H^{(3)}_{f+g}I\\+(H^{(2)}_{f+g}\circ H_gI-H_g\circ H^{(2)}_{f+g}I)
+(H_f\circ H_{[f,g]}I+H_{[f,g]}\circ H_gI)
\end{array}$$
and now since
$$\begin{array}{c}
H^{(2)}_{f+g}\circ H_gI-H_g\circ H^{(2)}_{f+g}I=H_{f+g}\circ H_{[f,g]}I+H_{[f,g]}\circ H_{f+g}I;\\
H_f\circ H_{[f,g]}I+H_{[f,g]}\circ H_gI=\hbox{ad}^2_fg+H_{[f,g]}\circ H_{f+g}I=\hbox{ad}^2_gf+H_{f+g}\circ H_{[f,g]}I
\end{array}$$
we finally obtain
$$\frac1{3!}(H_f\boxplus H_g)^3I=\frac1{3!}H^{(3)}_{f+g}I+\frac12\left(H_{f+g}\circ(\frac12H_{[f,g]}I)+(\frac12H_{[f,g]})\circ H_{f+g}I\right)+\frac1{12}(\hbox{ad}^2_fg+\hbox{ad}^2_gf).
$$
This confirms the third term in the expansion of the Baker-Campbell-Hausdorff formula.
\end{rem}
The following statement shows how to express a second order operator differently. 
\begin{prop} Let $f_1,\dots,f_m\in C^1(\R^n;\R^n)$ be available vector fields and $u\in C^2(\R^n)$. Then
$$\begin{array}{c}
(H_{f_1}\boxplus \dots\boxplus H_{f_m})^2I=D{(f_1+\dots+f_m)}(f_1+\dots+f_m)+\sum_{1\leq i<j\leq m}[f_i,f_j],\\
(H_{f_1}\boxplus \dots\boxplus H_{f_m})^2u=H^{(2)}_{f_1+\dots+f_m}u+\sum_{1\leq i<j\leq m}[f_i,f_j]\cdot \nabla u.
\end{array}$$
\end{prop}
\begin{proof}
We only prove the second formula.
We already know that
$$(H_{f_1}\boxplus H_{f_2})^2u=H^{(2)}_{f_1+f_2}u+[f_1,f_2]\cdot \nabla u.$$
Now notice that by definition
$$\begin{array}{ll}
(H_{f_1}\boxplus H_{f_2}\boxplus H_{f_3})^2u=(H_{f_1}\boxplus H_{f_2})^2u+H^{(2)}_{f_3}u+2H_{f_1+f_2}\circ H_{f_3}u\pm H_{f_3}\circ H_{f_1+f_2}u\\
=H^{(2)}_{f_1+f_2}u+H_{f_1+f_2}\circ H_{f_3}u+H_{f_3}\circ H_{f_1+f_2}u+H^{(2)}_{f_3}u+[f_1,f_2]\circ\nabla u+[f_1+f_2,f_3]\cdot\nabla u\\
=H^{(2)}_{f_1+f_2+f_3}u+([f_1,f_2]+[f_1,f_3]+[f_2,f_3])\cdot\nabla u.
\end{array}$$
We complete similarly the statement by induction. We leave the easy details to the reader.
\end{proof}

We add some definitions to the vector case.
\begin{defin}
We say that a function $u:\R^n\to\R^h$ has a $k-th$ order rate of change in the direction of $v\in\R^h$ for the control system at $x_o\in\R^n$ if there are available vector fields $f_1,\dots f_m$ such that
\begin{equation}\label{eqrateu}
v=\frac1{k!}(H_{f_1}\boxplus\dots\boxplus H_{f_m})^ku(x_0)\neq0,\;(H_{f_1}\boxplus\dots\boxplus H_{f_m})^iu(x_0)=0\;\hbox{for all }i=1,\dots,k-1.
\end{equation}
Let $\mathcal L_u:=\{v:v\hbox{ satisfies }(\ref{eqrateu})\}\subset\R^h$.

In particular we say that a balanced trajectory $x_s$ of (\ref{eqsystem}) as in (\ref{eqtwoswitch}) moves at $k-$th order rate in the direction of $v\in\R^n$ for the control system at $x_o\in\R^n$ if the vector fields $f_1,\dots f_m$ are such that
\begin{equation}\label{eqratei}
v=\frac1{k!}(H_{f_1}\boxplus\dots\boxplus H_{f_m})^kI(x_0)\neq0,\;(H_{f_1}\boxplus\dots\boxplus H_{f_m})^iI(x_0)=0\;\hbox{for all }i=1,\dots,k-1.\end{equation}
Let $\mathcal L\equiv \mathcal L_I:=\{v:v\hbox{ satisfies }(\ref{eqratei})\}\subset\R^n$.
\end{defin}
\begin{rem}\label{remsymmetric} {(Brackets and iterated Lie brackets.)}
Similarly to what we observed in (\ref{eqbracket}), if $f,g,h$ and $-f,-g,-h$ are available vector fields, in particular in the case of a symmetric system, then with a little algebraic effort one sees that, as expected,
$$(H_f\boxplus H_g\boxplus H_{-f}\boxplus H_{-g})I \equiv0,\quad(H_f\boxplus H_g\boxplus H_{-f}\boxplus H_{-g})^2I =2{[f,g]},$$
thus $[f,g]\in\mathcal L$ if it is nonvanishing at $x_o$, while on the other hand
$$\begin{array}{c}
(H_f\boxplus H_g\boxplus H_{-f}\boxplus H_{-g}\boxplus H_h\boxplus H_g\boxplus H_f\boxplus H_{-g}\boxplus H_{-f}\boxplus H_{-h})I\equiv 0,\\
(H_f\boxplus H_g\boxplus H_{-f}\boxplus H_{-g}\boxplus H_h\boxplus H_g\boxplus H_f\boxplus H_{-g}\boxplus H_{-f}\boxplus H_{-h})^2I\equiv0,\\
(H_f\boxplus H_g\boxplus H_{-f}\boxplus H_{-g}\boxplus H_h\boxplus H_g\boxplus H_f\boxplus H_{-g}\boxplus H_{-f}\boxplus H_{-h})^3I=6{[[f,g],h]},
\end{array}$$
so that $[[f,g],h]\in\mathcal L$, if it is nonvanishing at $x_o$.
Therefore, when using the identity function in $\R^n$, and if the system is symmetric, continuing in this fashion one can check that the complete Lie algebra generated by the available vector fields is contained in $\mathcal L$.
\end{rem}

\begin{rem}\label{remaffine} (The $ad$ operator for vector fields.)
Using the argument of Remark \ref{remequilibrium} we can see the following. Assume that $f_o(x)+\gep f_1(x)$ are available vector fields for all $\gep\in[-1,1]$ and that $f_o(x_o)=0$. This happens for instance in the case of an affine system with drift $f_o$ and $x_o$ as an equilibrium point, in particular for a linear system. Let $f(x)=f_o(x)+\gep f_1(x)$ and $g(x)=f_o(x)-\gep f_1(x)$. Then observe that by Remark \ref{remequilibrium}
$$(H_f\boxplus H_g)^{k+1}I(x_o)=2\gep(-1)^k\hbox{ad}^k_{f_o}f_1(x_o)+o(\gep),
$$
as $\gep\to0$. This indicates that we can find iterated Lie brackets given by the $ad$ operator close to directions in $\mathcal L$. Indeed we find for instance
$$\begin{array}{cc}
(H_f\boxplus H_g\boxplus H_g\boxplus H_f)I(x_o)=2f_o(x_o)=0, (H_f\boxplus H_g\boxplus H_g\boxplus H_f)^2I(x_o)=0,\\ 
(H_f\boxplus H_g\boxplus H_g\boxplus H_f)^3I(x_o)=12\gep\hbox{ad}^2_{f_o}f_1(x_o)+4\gep^2\hbox{ad}^2_{f_1}f_o(x_o).
\end{array}$$
Therefore given $v$ a unit vector, if $v\cdot \hbox{ad}^2_{f_o}f_1(x_o)<0$ then for $\gep$ sufficiently close to 0 also $v\cdot (H_f\boxplus H_g\boxplus H_g\boxplus H_f)^3I(x_o)<0$ and $(1/k!)(H_f\boxplus H_g\boxplus H_g\boxplus H_f)^3I(x_o)\in\mathcal L$.
We can proceed further as
$$(H_f\boxplus H_g\boxplus H_g\boxplus H_f\boxplus H_g\boxplus H_f\boxplus H_f\boxplus H_g)^jI(x_o)=0
$$
for $j=1,2,3$ while $(H_f\boxplus H_g\boxplus H_g\boxplus H_f\boxplus H_g\boxplus H_f\boxplus H_f\boxplus H_g)^4I(x_o)=-192\gep \hbox{ad}^3_{f_o}f_1(x_o)+o(\gep)$ and so forth.
\end{rem}

\section{Small time local attainability of fat targets}

In this section we obtain a first attainability result for the control system (\ref{eqsystem}) with the general assumptions of the previous Section 2.1, in the case of targets locally described by an inequality. Namely in the neighborhood $B_{R}(x_o)$ of $x_o$ the target is described as 
$${\mathcal T}=\{x:u(x)\leq u(x_o)\},$$
where $u\in C^1(\R^n)$, $\nabla u(x_o)\neq0$.
This case is easier because $u$ is scalar as there is only one constraint to cope with and we do not really need the Lemma in the Appendix.
Later in the section we also consider some classes of nonsmooth targets.
\begin{thm}\label{teofat}
For some  integer $k>0$ let $u\in C^{k}(\R^n)$ and $x_o\in\mathcal T\backslash\hbox{int}\mathcal T$ be such that $\nabla u(x_o)\neq0$. Suppose that $u$ has $k-$th order decrease rate, i.e. there are available vector fields $f_i\in C^{k-1}(\R^n;\R^n)$, $i=1,\dots,m$ such that
\begin{equation}\label{eqassfat}
(H_{f_1}\boxplus\dots\boxplus H_{f_m})^ku(x_o)<0\hbox{ and }(H_{f_1}\boxplus\dots\boxplus H_{f_m})^ju(x_o)=0 \hbox{ for }j=1,\dots,k-1.
\end{equation}
Then there are $\delta,\delta'>0$ and a constant $K>0$ such that for any $x\in\R^n$, $|x-x_o|\leq \delta'$ we can find $t\in[0,\delta]$ such that if $x_s$ is the trajectory obtained by using the vector  fields $f_i$ on subsequent intervals of length $t$ as in (\ref{eqtwoswitch}) and starting out at $x$, then $x_{mt}\in B_R(x_o)$, $u(x_{mt})\leq u(x_o)$
and moreover
$$t\leq K|x-x_o|^{1/k}.$$

In particular the minimum time function $T$ to reach the target $\mathcal T$ satisfies 
\begin{equation}\label{eqmtfest}
T(x)\leq Km|x-x_o|^{1/k},\quad x\in B_{\delta'}(x_o)\end{equation}
and thus it is continuous at $x_o$ and the target is STLA at $x_o$.
\end{thm}
\begin{proof} The proof is local, in the neighborhood of the point $x_o$.
By the assumptions, we can choose $c_o>0$ such that $A:=(H_{f_1}\boxplus\dots\boxplus H_{f_m})^ku(x_o)\leq -2c_o<0$. We pick a radius $R>0$ so that $\nabla u(x)\neq0$ in $B_R(x_o)$ and $\mathcal T\cap B_R(x_o)=\{x\in B_R(x_o):u(x)\leq u(x_o)\}$. We consider $\sigma>0$ so that for any $|x-x_o|\leq R/2$, $0\leq t\leq\sigma$, any trajectory $\{x_s:s\geq0\}$ of the control system using only the vector fields $f_i$ and starting at $x$ satisfies $|x_s-x_o|\leq R$, for $s\leq mt$.

For any $x,t$ such that $|x-x_o|\leq R/2$, $0\leq t\leq\sigma$, and $u(x)> u(x_o)$, we construct the trajectory $x_s$ using the vector fields $f_1,\dots,f_m$ in subsequent  intervals of length $t$ and starting out at $x$, and the corresponding reference trajectory $x_s^o$ starting at $x_o$ and using the same control. Notice that changing $t$ modifies the trajectory drastically.
We define the continuous function (see Remark \ref{remneed})
$$\rho(x,t)=u(x^o_{mt})-u(x_{mt})$$
and we observe that by local Lipschitz continuity of $u$ and Gronwall estimates on the trajectories, there are constants $C,L>0$ such that 
\begin{equation}\label{eqrhoest}
|\rho(x,t)|\leq \hat C|x_{mt}-x^o_{mt}|\leq C|x-x_o|e^{Lm\sigma},
\end{equation}
 for all $|x-x_o|\leq R/2$ and $t\in[0,\sigma]$. In particular
\begin{equation}\label{eqlemfunct1}
\lim_{t\to0}\rho(x,t)=u(x^o)-u(x)<0,\quad
\lim_{x\to x_o}\sup_{t\in [0,\sigma]}|\rho(x,t)|=0.
\end{equation}
By the asymptotic formula of the trajectory $x^o_s$ proven in Proposition \ref{proptaylor} and the assumptions, 
we have that
$$u(x^o_{mt})-u(x_o)=\sum_{j=1}^{k}\frac{t^j}{j!}(H_{f_1}\boxplus\dots\boxplus H_{f_m})^ju(x_o)+\frac{t^k}{k!}\gamma(t)=\frac{t^k}{k!}\left((H_{f_1}\boxplus\dots\boxplus H_{f_m})^ku(x_o)+\gamma(t)\right),$$
where $\lim_{t\to0}\gamma(t)=0$. 
Therefore we can find $0<\delta(\leq\sigma)$ independent of $x$, such that 
$|\gamma(t)|\leq c_o,\hbox{ for all } t\in[0,\delta]$.
For a given $x$, $|x-x_o|\leq R/2$, we thus prove that by (\ref{eqrhoest})
\begin{equation}\label{eqbb}
\begin{array}{ll}
u(x_{mt})-u(x_o)&=u(x_{mt})-u(x^o_{mt})+\frac{t^k}{k!}\left((H_{f_1}\boxplus\dots\boxplus H_{f_m})^ku(x_o)+\gamma(t)\right)
\\&
\leq \frac{t^k}{k!}\left(A+\gamma(t)\right)-\rho(x,t)\leq
-c_ot^k/k!+C|x-x_o|e^{Lm\sigma}.
\end{array}\end{equation}
We now notice that the right hand side in (\ref{eqbb}) is zero for
$$t=t^*=\left(\frac{Ck!}{c_o}e^{Lm\sigma}\right)^{1/k}|x-x_o|^{1/k}$$
and that this choice is admissible since $t^*\leq \delta$ provided we choose 
$$|x-x_o|\leq \delta'=\frac{c_o}{Ck!}e^{-Lm\sigma}\delta^k.$$

Thus the target is reached at most at time $mt^*$ and the estimate on the minimum time function holds.
\end{proof}
If the decrease rate sufficient condition can be satisfied in the viscosity solutions sense, then we can drop regularity of the target.
\begin{cor}\label{corfat}
Let $u\in C(\R^n)$ and $x_o\in\mathcal T\backslash\hbox{int}\mathcal T$. Suppose that there is $\Phi\in C^k(\R^n)$ such that $u-\Phi$ attains a local maximum point at $x_o$, $\nabla \Phi(x_o)\neq0$ and $\Phi$ has $k-$th order decrease rate at $x_o$. Then all conclusions of Theorem \ref{teofat} hold true.
\end{cor}
\begin{proof}
If there is a neighborhood of $x_o$ such that $u(x)-\Phi(x)\leq u(x_o)-\Phi(x_o)$ for $x\in B_R(x_o)$, then
$$\hat{\mathcal T}=\{x\in B_R(x_o):\Phi(x)\leq\Phi(x_o)\}\subset \{x\in B_R(x_o):u(x)\leq u(x_o)\}.$$
We thus apply the assumption on $\Phi$ and determine that $\hat{\mathcal T}$, and therefore $\mathcal T$, is STLA at $x_o$.
\end{proof}
\begin{rem}
In our previous papers \cite{so1,so2} we proved that if $k=2$ and the system is symmetric or affine, then the decrease rate condition of the previous Theorem with $m\leq2$, e.g. a trajectory with at most one switch, is necessary and sufficient for the target being STLA at $x_o$ and holding (\ref{eqestt}), the necessary part being proven in \cite{bfs}.
\end{rem}
The next statement shows how to improve in a simple way estimate (\ref{eqmtfest}) and obtain H\"older regularity of the minimum time function. We need some uniformity of the pointwise estimate satisfied by $T$. Below we indicate the distance function from the target as
$d(x,\mathcal T):=\min\{|x-y|:y\in\mathcal T\}.$ Of course the next statement can be generalized as in Corollary \ref{corfat}.
\begin{prop}
We consider the control system (\ref{eqsystem}) and the target $\mathcal T$ at the beginning of the section with $u\in C^2(\R^n)$. Assume that 
there are $R,C>0$ and an integer $k$ such that for all $x\in \mathcal T\cap B_R(x_o)$ the minimum time function satisfies 
$$T(y)\leq C|y-x|^{1/k}\quad\hbox{for all }y\in B_{R/2}(x).$$
Then $T$ also satisfies
$$T(x)\leq d(x,\mathcal T)^{1/k},\quad\hbox{for all }x\in B_{R/2}(x_o).$$
In particular if the assumption holds for all $x_o\in\mathcal T$, then $T$ is locally $1/k-$H\"older continuous in its domain if $F$ satisfies (\ref{eqlip}) globally.
\end{prop}
\begin{proof}
Let $x\in B_{R/2}(x_o)$, then by regularity of $u$ there is a unique projection $p_x\in\mathcal T$, $|x-p_x|=d(x,\mathcal T)$ and $p_x\in  B_R(x_o)$ since $|p_x-x_o|\leq|p_x-x|+|x-x_o|\leq2|x-x_o|$. Then we can apply the assumption at $p_x$ and find
$$T(x)\leq C|x-p_x|^{1/k}=Cd(x,\mathcal T)^{1/k}.$$
The last part of the statement comes from a standard argument as in the book \cite{bcd} or in \cite{bfs}.
\end{proof}
The next consequence deals with another class of nonsmooth sets.
\begin{cor}
For some  integer $k>0$ let $u_i\in C^{k}(\R^n)$, $i=1,\dots,l$, be such that $\nabla u_i(x_o)\neq0$, and $\mathcal T\cap B_R(x_o)=\{x\in B_R(x_o):u_i(x)\leq u_i(x_o),\;i=1,\dots,l\}$. Suppose that $u_i$ has $k_i-$th order decrease rate for all $i$, i.e. there are available vector fields $f_h\in C^{k-1}(\R^n;\R^n)$, $h=1,\dots,m$ such that
\begin{equation}\label{eqassman}
(H_{f_1}\boxplus\dots\boxplus H_{f_m})^{k_i}u_i(x_o)<0\hbox{ and }(H_{f_1}\boxplus\dots\boxplus H_{f_m})^ju_i(x_o)=0 \hbox{ for }j=1,\dots,k_i-1.
\end{equation}
Then there are $\delta,\delta'>0$ and a constant $K>0$ such that for any $x\in\R^n$, $|x-x_0|\leq \delta'$ we can find $t\in[0,\delta]$ such that if $x_s$ is the trajectory obtained by using the vector  fields $f_i$ as in (\ref{eqtwoswitch}) and starting out at $x$, then $u_i(x_{mt})\leq u_i(x_o)$ for all $i=1,\dots,l$ and moreover
$t\leq K|x-x_o|^{1/k}$ if $k=\max_ik_i$.
All other conclusions of Theorem \ref{teofat} remain unchanged.
\end{cor}
\begin{proof} We modify the proof of Theorem \ref{teofat}.
By the assumptions, we can choose $c_o>0$ such that $A_i:=(H_{f_1}\boxplus\dots\boxplus H_{f_m})^ku_i(x_o)\leq -2c_o<0$, for all $i=1,\dots,l$. Following the proof of Theorem \ref{teofat}, we reach (\ref{eqbb}) for each constraint, and if $\delta\leq1$,
$$u_i(x_{mt})-u_i(x_o)\leq -c_o\frac{t^{k_i}}{k_i!}
-\rho_i(x,t)\leq -c_o\frac{t^k}{k!}
-\rho(x,t),$$
where 
$=1,\dots,l$, $\rho(x,t)=\min_i\rho_i(x,t)$. We conclude as in Theorem \ref{teofat}.
\end{proof}

\section{Small time local attainability of manifolds}

\subsection{The case of a point}

In this section the target for system (\ref{eqsystem}) is a point $\mathcal T=\{x_o\}\subset\R^n$, which is identified by the system $u(x)=x-x_o=0$.
This case simplifies with respect to the general one because the target is determined by {\it flat} constraints. The constants $R,\sigma$ below will follow Section 2.1.

We start with a useful Lemma comparing a trajectory at a point and the translation of the trajectory starting at a different point but following the same control.
\begin{lem}\label{lema1}
Let $(x^o_s)_{s\in[0,t]}$ be a solution of (\ref{eqsystem}) starting at $x_o$ and $(a_s)_{s\in[0,t]}$ be the corresponding control. Let $x\neq x_o$ and $y_s=x^o_s+x-x_o$, $s\in[0,t]$ be the translation of $x^o_s$ starting at $x$. Then if $t,|x-x_o|$ are sufficiently small, the trajectory $(x_s)_{s\in[0,t]}$ of (\ref{eqsystem}) starting at $x$ with control $a_s$ satisfies
$$| x_t-y_t|\leq Le^{Lt}|x-x_o|t.$$
\end{lem}
\begin{proof}
The translated trajectory $y_s$ is itself a trajectory of a translated control system
$$\dot y_s=F(y_s+x_o-x,a_s),\quad y_0=x.$$
Therefore
$$| x_s-y_s|\leq \int_0^s|F( x_r,a_r)-F(y_r+x_o-x,a_r)|\;dr\leq L|x-x_o|s+L\int_0^s| x_r-y_r|\;dr
$$
for all $s\in[0,t]$. By usual Gronwall estimates we then get
$$| x_t-y_t|\leq Le^{Lt}|x-x_o|t.$$
\end{proof}

In this section we are going to assume the following condition at $x_o$:
\begin{itemize}
\item[(A1)]{We have $m$ groups of available vector fields of the control system: $f^{(i)}_1,\dots,f^{(i)}_{j(i)}$, $i=1,\dots,m$ and integers $k_i>0$ such that $f^{(i)}_r\in C^{k_i-1}(\R^n;\R^n)$ and we denote $\bar k=\max_ik_i$. We suppose for convenience that $j(i)=1$ if $k_i=1$. We assume that
$$(H_{f^{(i)}_1}\boxplus\dots\boxplus H_{f^{(i)}_{j(i)}})^rI(x_o)=0,\quad 1\leq r<k_i,\;i=1,\dots,m.
$$
and 
$$\R^n\ni A^o_i:=\frac1{k_i!}(H_{f^{(i)}_1}\boxplus\dots\boxplus H_{f^{(i)}_{j(i)}})^{k_i}I(x_o)\neq0.
$$}
\end{itemize}
In particular the $i-$th group has $j(i)$ vector fields and the trajectory moves at $k_i-$th order rate at $x_o$ in the direction of $A^o_i\in\mathcal L$.
We construct the $n\times m$ matrix, written in columns as
$A_o=\left(A^o_i
\right)_{i=1,\dots,m}$,
so that $A_o$ has a column for each group of vector fields. 
The main assumption on matrix $A_o$ is the following.
\begin{itemize}
\item[(A2)]{As a $n\times m$ matrix, the $m$ columns of the matrix $A_o$ are a {\it positive basis} of $\R^n$. In particular $A_o$ has rank $n$ (and $m\geq n+1$).}
\end{itemize}
\begin{rem} The notion of positive basis of a vector space is classical and recalled in the Appendix together with a technical lemma that we need in the main results of this section, where we use the assumption (A2). If $k_i=1$ for each $i$, then in (A1) we have $m$ vector fields satisfying $A^o_i=f_i(x_o)\neq0$. In this case assumption (A2) is the positive basis considion of the classical work by Petrov \cite{pe,pe2}. If $\bar k=2$, then (A2) has been used by Liverovskii \cite{li,li2} as necessary and sufficient second order condition when the columns of $A_o$ are either available vector fields or their first Lie brackets.
\end{rem}
Below for any vector $\tau\in\R^m$ we indicate $\tau\geq0$ if $\tau=(\tau_1,\dots,\tau_m)$ and $\tau_i\geq0$ for $i=1,\dots,m$. Let $\tau=(\tau_1,\dots,\tau_m)\geq0$.
We are going to define several reference trajectories. For $i=1,\dots,m$ and $\tau_i\geq0$, trajectory $x^{o,i}_t$ starts at $x_o$ and it is a balanced trajectory of the $j(i)$ fields $f^{(i)}_1,\dots,f^{(i)}_{j(i)}$ followed in consecutive time intervals of length $\tau_i^{1/k_i}$ each, and therefore by Proposition \ref{proptaylorvector} we know that at the end point
\begin{equation}\label{eqfirstestimate}
x^{o,i}_{j(i)\tau_i^{1/k_i}}=x_o+A^o_i\tau_i+\tau_io(1),
\end{equation}
as $\tau_i\to0$.
For the given family of nonnegative times in the vector $\tau\geq0$ we build recursively a further reference trajectory as follows. We put $T_0=0$ for convenience.
For the first group: we consider the trajectory 
$x^o_s\equiv x^{o,1}_s$, $s\in[0,T_1]$, $T_1=j(1)\tau_1^{1/k_1}$.
We proceed recursively with the following groups of vector fields. If we have defined $x^o_s$ up to the $i-$th group and time $T_i=\sum_{l=1}^i j(l)\tau_l^{1/k_l}$,
we proceed with the $(i+1)$-th group of vector fields starting at $x^o_{T_i}$ and prolonging the trajectory by following the vector fields $f^{(i+1)}_1,\dots,f^{(i+1)}_{j(i+1)}$ in $j(i+1)$ successive intervals of respective length $\tau_{i+1}^{1/k_{i+1}}$ until we reach the point $x^o_{T_{i+1}}$, $T_{i+1}=T_i+j({i+1})\tau_{i+1}^{1/k_{i+1}}$.
Recursively, for all $\tau\geq0$ we have a trajectory $x^o_s$ well defined for $s\in[0,T_m]$.

For a generic initial point $x\in B_{R/2}(x_o)$, and $s\in[0,T_m]$, we also consider the corresponding  trajectory $x_s$ starting at $x$ with the same control as $x^o_s$. We finally construct the translated trajectories $y^{i+1}_s=x^{o,i+1}_s+x_{T_{i}}^o-x_o$ for $s\in[0,j(i+1)\tau_{i+1}^{1/k_{i+1}}]$. Notice that $y^{i+1}_0=x_{T_{i}}^o$. We have now defined $x^{o,i}_s,y^i_s$, $s\in[0,T_i-T_{i-1}]$, $i=1,\dots,m$, and $x^o_s,x_s$, $s\in[0,T_m]$.

We start with the following Lemma that we will also use in the following section.
\begin{lem}\label{lemestimate}
Consider a nonlinear control system in the form (\ref{eqsystem}) and suppose that (A1) is satisfied at the point $x_o\in\mathcal T$. Then with the notations above
\begin{equation}\label{eqindu}
x_{T_i}^o=x_o+\sum_{j=1}^iA^o_j\tau_j+(\tau_1+\dots+\tau_i)o(1),
\end{equation}
for all $i=1,\dots,m$, as $\tau\to0$. In particular there is $C>0$ such that
$$|x_{T_i}^o-x_o|\leq C (\tau_1+\dots+\tau_i),$$
for $i=1,\dots,m$ and all $\tau$ sufficiently small.
\end{lem}
\begin{proof}
We prove the statement by an induction argument. For the first group of vector fields and (\ref{eqfirstestimate}) we have
$$x_{T_1}^o=x_{T_1}^{o,1}=x_o+A^o_1\tau_1+\tau_1o(1),
$$
as $\tau\to0$. 
Suppose now by induction that after the i-th group of vector fields, $1\leq i<m$, we have that (\ref{eqindu}) is satisfied. Observe that
$$\begin{array}{ll}
x_{T_{i+1}}^o-x_o&
=(x_{T_{i+1}}^o-y_{T_{i+1}-T_i}^{i+1})+(x_{T_{i+1}-T_i}^{o,i+1}-x_o)+(x_{T_{i}}^o-x_o)\\
&=(x_{T_{i+1}}^o-y_{T_{i+1}-T_i}^{i+1})+A^o_{i+1}\tau_{i+1}+\tau_{i+1}o(1)+\sum_{j=1}^iA^o_j\tau_j+(\tau_1+\dots+\tau_i)o(1)
\end{array}$$
where we rewrote the second term in the first line by (\ref{eqfirstestimate}) and the third by the induction assumption (\ref{eqindu}) at $i$. Notice that by Lemma \ref{lema1}, if $T_m\leq\sigma$, comparing the trajectories in $[T_i,T_{i+1}]$,
$$|x_{T_{i+1}}^o-y_{T_{i+1}-T_i}^{i+1}|\leq C|x_{T_{i}}^o-x_o|(T_{i+1}-T_i)\leq \hat C(\tau_1+\dots+\tau_{i})o(1),
$$
again by the induction assumption (\ref{eqindu}). This proves that (\ref{eqindu}) holds for $i+1$ and the conclusion is thus reached.
\end{proof}
We now prove the main result of this part of the paper.
\begin{thm}\label{teopoint} Consider a nonlinear control system in the form (\ref{eqsystem}) and suppose that (A1) and (A2) are satisfied at the point $x_o\in\R^n$. Then the target ${\mathcal T}=\{x_o\}$ is STLA and the minimum time function $T(x)$ at a point $x$ in the neighborhood of $x_o$ satisfies
$$T(x)\leq C|x-x_o|^{1/\bar k}.$$
In particular $T$ is locally $1/\bar k$-H\"older continuous in its domain, if $F$ satisfies (\ref{eqlip}) globally.
\end{thm}
\begin{proof}
Given $x\neq x_o$, $x\in B_{R/2}(x_o)$, we want to find $\R^m\ni\tau\geq0$ so that $x_{T_m}=x_o$, where $(x_s)_s$ has been defined above. Notice that by Gronwall estimates on trajectories and if $T_m\leq \sigma$, therefore for $\tau\geq0$ sufficiently small, then the continuous function $\rho(x,\tau)=x_{T_m}^o-x_{T_m}$ satisfies
$$|\rho(x,\tau)|\leq C|x-x_o|.$$
In particular $\lim_{x\to x_o}\sup_{|\tau|\leq\sigma}|\rho(x,\tau)|=0$.
We use 
 Lemma \ref{lemestimate} for $i=m$ and conclude that
\begin{equation}\label{eqfindu}
x_{T_{m}}-x_o=x_{T_{m}}-x_{T_{m}}^o+x_{T_{m}}^o-x_o=A_o\tau+\hat \gamma(\tau)\sum_{i=1}^m\tau_i-\rho(x,\tau)=(A_o+\gamma(\tau))\tau-\rho(x,\tau),
\end{equation}
where $\hat\gamma$ is a continuous function with values in $\R^n$ such that $\hat \gamma(0)=0$
and $\gamma(\tau)$ is $n\times m$ matrix valued vanishing as $\tau\to0$ with $m$ columns all equal to $\hat\gamma$. Finding $\tau\geq0$ so that the right hand side of (\ref{eqfindu}) is zero is the content of Lemma \ref{lempetrovter}(i) in the Appendix. It uses a fixed point argument and it is where the assumption (A2) is finally needed. Thus there are $\delta,\delta'>0$ such that for all $x,|x-x_o|<\delta'$ we can find $\tau\geq0$, $|\tau|<\delta$ and for such $\tau$, $x_{T_m}=x_o$. The corresponding trajectory $x_s$ then reaches the point $x_o$ at time $T_m$. Moreover Lemma \ref{lempetrovter} also shows that for the specific $\tau$ we have
$$\sum_{i=1}^m\tau_i\leq K\sup_{|\tau|\leq\delta}|\rho(x,\tau)|\leq C|x-x_o|$$
and therefore (we may assume that $\tau_i\leq1$ for all $i$)
$$T_m=\sum_{i=1}^mj(i)\tau_i^{1/k_i}\leq \hat C(\sum_{i=1}^m\tau_i^{1/\bar k})\leq \tilde C|x-x_o|^{1/\bar k}.
$$
The final statement of the Theorem concerning H\"older regularity of the minimum time function now follows by well known arguments, see e.g. \cite{bcd}.
\end{proof}
\begin{rem}
In the case of a point as target, Remark \ref{remsymmetric} applied to symmetric systems outlines the fact that the set $\mathcal L\supset\{h(x_o):h\in\hbox{Lie}(F)\}=L(F)(x_o)$, where Lie$(F)$ is the Lie algebra generated by all the available vector fields of the control system. If $L(F)(x_o)=\R^n$, then $\mathcal L$ contains a positive basis on $\R^n$ and we can recover in Theorem \ref{teopoint} the classical sufficient condition for small time local controllability of Rashevski and Chow \cite{cho}.\\
Similarly Remark \ref{remaffine} applied to the case of affine systems with an equilibrium point at $x_o$ outlines the fact that $L^a(F)(x_o)=\cup_{\lambda\geq0}\lambda\;\hbox{co}\{ad^k_{f_o}f(x_o):k\in\mathbb N,\,F(x,a)=f_o(x)+f(x),a\in A\}\subset \cup_{\lambda\geq0}\lambda\;\hbox{co}\mathcal L$. Here $f_o$ is the drift, $f_o(x_o)=0$, $f$ is a generic available vector field in the symmetric part and co$C$ is the convex hull of the set $C$. If $L^a(F)(x_o)=\R^n$ then the set $\mathcal L$ contains a positive basis of $\R^n$. Therefore in Theorem \ref{teopoint} we also recover the sufficient condition of Frankovska \cite{fra1} and Kawski \cite{kaw}.
\end{rem}

\subsection{A general manifold with a boundary}

In this section we will discuss the attainability of general smooth targets, namely manifolds possibly with a boundary. We therefore assume that in the neighborhood of the point $x_o$ the target is described by a set of equations and at most one inequality. Let $u=\;^t(u_1,\dots,u_h):\R^n\to\R^h$ and possibly also $u_{h+1}:\R^n\to\R$ be at least of class $C^1$ and such that $\nabla u_i(x_o)$, $i=1,\dots,h,(h+1)$, are linearly independent. The previous assumption stands throughout the section.
Thus the target is locally defined in the neighborhood of $x_o$ in one of the two following ways
\begin{equation}\label{eqvincolo}\begin{array}{c}
{\mathcal T}=\{x:u(x)=u(x_o)\},\\
{\mathcal T}=\{x:u(x)=u(x_o),\;u_{h+1}(x)\geq u_{h+1}(x_o)\}.
\end{array}\end{equation}
In the second case $\mathcal T$ has a boundary (as a manifold) 
$$\partial{\mathcal T}=\{x:u(x)=u(x_o),\;u_{h+1}(x)= u_{h+1}(x_o)\}.$$
We will suppose for convenience that $1\leq h\leq n-1$, as the two cases $h=0$, $h=n$ have been previously treated. In particular the target has no interior points.
We will keep all notations as in the previous Section 4.1. The constants $R,\sigma,L,M$ will also follow Section 2.1.
Next we compare the variation of a given vector valued smooth function.
\begin{lem}\label{lema2}
(i) Let $(x^o_s)_{s\in[0,t]}$ be a solution of (\ref{eqsystem}) starting at $x_o$ and $(a_s)_{s\in[0,t]}$ be the corresponding control. Let $(x_s)_{s\in[0,t]}$ be the trajectory of (\ref{eqsystem}) starting at $x\in B_{R/2}(x_o)$ with the same control $a_s$ and let $u:\R^n\to\R^h$, $u=(u_1,\dots,u_h)$, be a function of class $C^2$. Then, if $t,|x-x_o|>0$ are sufficiently small,
\begin{equation}\label{eqestu}
u(x_t)-u(x)=u(x^o_t)-u(x_o)+\alpha(x,t),
\end{equation}
where $\alpha:B_{R/2}(x_o)\times [0,\sigma]\to \R^h$ is a continuous function and $|\alpha(x,t)|\leq C|x-x_o|t$.

(ii) Suppose moreover that $B_R(x_o)\times A\ni (x,a)\mapsto Du(x)F(x,a)$ only depends on a restricted group of variables $x_{l_1},\dots,x_{l_p}$ and $a\in A$, and that
the corresponding components of the vector field $F_{l_1},\dots,F_{l_p}$ also depend only on the same group of variables as well.
Suppose moreover
\begin{equation}\label{eqassrestrict}
|Du(x)F(x,a)-Du(y)F(y,a)|\leq \hat L\sqrt{\sum_{j=1}^p(x_{l_j}-y_{l_j})^2},
\end{equation}
for all $x,y\in B_R(x_o)$ and $a\in A$.
Then (\ref{eqestu}) holds with $\alpha$ satisfying the stronger estimate
$$|\alpha(x,t)|\leq Ct\sqrt{\sum_{j=1}^p(x_{l_j}-(x_o)_{l_j})^2}.$$
\end{lem}
\begin{proof}
Note that (\ref{eqestu}) holds with
$$\alpha(x,t)=u(x_t)-u(x)-(u(x^o_t)-u(x_o))=\int_0^t\left(Du(x_s) F(x_s,a_s)-Du(x^o_s) F(x^o_s,a_s)\right)\;ds$$
and then by Gronwall inequality
$$\begin{array}{ll}
|\alpha(x,t)|&\leq \int_0^t(\|D^2u\|_\infty M+\|Du\|_\infty L)|x_s-x_s^o|\;ds\\
&\leq (\|D^2u\|_\infty M+\|Du\|_\infty L)|x-x_o|\int_0^te^{Ls}\;ds\leq C|x-x_o|t,
\end{array}$$
for $|x-x_o|$ sufficiently small, $t\in[0,\sigma]$ and $C$ depending only on the data and $\sigma$.

The proof of (ii) is similar 
by considering the subsystem of (\ref{eqsystem}) of the group of components $F_{l_1},\dots,F_{l_p}$ and using (\ref{eqassrestrict}). \end{proof}

We are now going to assume the following conditions:
\begin{itemize}
\item[(B1)]{There are $m$ groups of available vector fields: $f^{(i)}_1,\dots,f^{(i)}_{j(i)}$, $i=1,\dots,m$ and integers $k_i>0$ such that $f^{(i)}_r\in C^{k_i-1}(\R^n;\R^n)$ and $u\in C^{\bar k}(\R^n;\R^h)$, if $x_o\in\mathcal T\backslash\partial \mathcal T$ (respectively $\hat u=(u,u_{h+1})\in C^{\bar k}(\R^n;\R^{h+1})$, if $x_o\in\partial\mathcal T$), where $\bar k=\max_ik_i$. We suppose for convenience that $j(i)=1$ if $k_i=1$. We assume that
$$\begin{array}{c}
(H_{f^{(i)}_1}\boxplus\dots\boxplus H_{f^{(i)}_{j(i)}})^ru_l(x_o)=0,\quad 0\leq r<k_i,\;i=1,\dots,m,\;l=1,\dots,h\hbox{ (resp. }h+1),\\
\R^h\ni k_i!A^o_i=(H_{f^{(i)}_1}\boxplus\dots\boxplus H_{f^{(i)}_{j(i)}})^{k_i}u(x_o)\neq0,\\
 (\hbox{resp. }
\R^{h+1}\ni k_i!\;^t(A^o_i,s_i)=(H_{f^{(i)}_1}\boxplus\dots\boxplus H_{f^{(i)}_{j(i)}})^{k_i}\hat u(x_o)\neq0),
\end{array}$$
i.e. $u$ has $k_i-$th order rate of change in the direction of $A^o_i\in\mathcal L_u$ (resp. $(A^o_i,s_i)\in\mathcal L_{\hat u}$).}
\end{itemize}
If $x_o\in\mathcal T\backslash\partial\mathcal T$,
we construct a $h\times m$ matrix, written in columns as
$A_o=\left(A_i^o\right)_{i=1,\dots,m}$ or else if $x_o\in\partial\mathcal T$,
we add to $A_o$ an extra row 
$$\hat A_o=\left(\begin{array}{c}A_o\\s\end{array}\right),
\quad s=\left(\frac1{k_i!}(H_{f^{(i)}_1}\boxplus\dots\boxplus H_{f^{(i)}_{j(i)}})^{k_i}u_{h+1}(x_o)\right)_{i=1,\dots,m}.$$ 
In particular in $A_o$ (respectively $\hat A_o$) there is only a column for each group of vector fields and it is not zero. We suppose that:
\begin{itemize}
\item[(B2)]{If $x_o\in\mathcal T\backslash\partial\mathcal T$, the $m$ columns of the matrix $A_o$ form a {\it positive basis} of $\R^h$. In particular $A_o$ has rank $h$. If instead $x_o\in\partial\mathcal T$ then the following holds: the matrix $\hat A_o$ has rank $h+1$ and
for all $p\in \R^h$ and $r\geq0$ there exists $\lambda=\;^t(\lambda_1,\dots,\lambda_m)\geq0$ such that
$p=A_o\lambda$, $r\leq s\cdot\lambda$.
}
\end{itemize}
The second part of assumption (B2) modifies the positive basis condition in order to be useful at boundary points of the manifold.

For any given family of nonnegative times $\tau=(\tau_1,\dots,\tau_m)\geq0$ sufficiently small, and an initial point $x\in B_{R/2}(x_o)$, we build as in the previous Section 6.1, the trajectories $x^o_s,\;x_s$ for $s\in[0,T_m]$ and $x^{o,i}_s$, for $s\in[0,j(i)\tau_i^{1/k_i}]$.
The next statement uses an additional assumption to (B1-2), in order to prove that the target is STLA at $x_o$. This looks strong in some examples because it requires first order and higher order terms to vanish independently.
\begin{thm}\label{teomanifold} Consider a nonlinear control system in the form (\ref{eqsystem}) and $x_o\in\mathcal T$, where the target $\mathcal T$ is locally described as above in the section. Suppose that (B1) and (B2) are satisfied at $x_o$. 
In addition we require that the vector fields satisfying (B1) also satisfy
$$(H_{f_1^{(i)}}\boxplus \dots \boxplus H_{f_{j(i)}^{(i)}})^rI(x_o)=0,
$$
for all $i=1,\dots,m$ and $1\leq r<k_i$, where $k_i$ comes from assumption (B1).
Then the target is STLA at $x_o$ and the minimum time function $T(x)$ at a point $x$ in the neighborhood of $x_o$ satisfies (for appropriate $\delta'\leq\sigma$)
\begin{equation}\label{eqtime}
T(x)\leq C|x-x_o|^{1/\bar k}\quad x\in B_{\delta'}(x_o).\end{equation}
If moreover (\ref{eqtime}) holds near all points $\hat x\in B_{\delta'}(x_o)\cap\mathcal T$ and $C,\delta',\bar k$ are independent of $\hat x$, then (\ref{eqtime}) improves to
$$T(x)\leq Cd(x,\mathcal T)^{1/\bar k}.
$$
If in particular the assumptions hold at all $x_o\in\mathcal T$ then $T$ is locally $1/\bar k-$ H\"older continuous in its domain, if $F$ satisfies (\ref{eqlip}) globally.
\end{thm}
\begin{proof}
We proceed similarly to the proofs of Theorem \ref{teopoint} and Lemma \ref{lemestimate}, so we only point out the main differences. With the notations of those results, for any $x\in B_{R/2}(x_o)$, we want to find a nonnegative vector $\tau\in\R^m$ so that $u(x_{T_m})=u(x_o)$ and in addition $u_{h+1}(x_{T_m})\geq u_{h+1}(x_o)$ when $x_o\in\partial\mathcal T$. Now, similarly to (\ref{eqfindu}),
$$u(x_{T_m})-u(x_o)=-\rho(x,\tau)+u(x_{T_m}^o)-u(x_o),$$
and again $\rho$ is continuous and $|\rho(x,\tau)|\leq C|x-x_o|$, by the local Lipschitz continuity of $u$, if $\tau$ is sufficiently small.
By the assumption (B1) we know that
\begin{equation}\label{eqaaexpand}
u(x^{o,i}_{j(i)\tau_i^{1/k_i}})-u(x_o)=A^o_i\tau_i+\tau_io(1),
\end{equation}
as $\tau\to0$, for $i=1\dots m$, since $A^o_i\in \mathcal L_u$. 
Moreover we also have
$$u_{h+1}(x^{o,i}_{k_i\tau_i^{1/k_i}})-u_{h+1}(x_o)=s_i\tau_i+\tau_io(1),$$
if $x_o\in\partial\mathcal T$.
We will proceed in the case $x_o\in\mathcal T\backslash\partial\mathcal T$ and modify accordingly at the end if $x_o\in\partial\mathcal T$.
We want to proceed recursively as in Lemma \ref{lemestimate} and assume that after the $i-$th step
\begin{equation}\label{eqindumanifold}
u(x_{T_i}^o)-u(x_o)=\sum_{j=1}^iA^o_j\tau_j+(\tau_1+\dots+\tau_i)o(1).
\end{equation}
Observe that (\ref{eqindumanifold}) holds for $i=1$ by (\ref{eqaaexpand}). Given (\ref{eqindumanifold})
\begin{equation}\label{eqaabbcc}\begin{array}{ll}
u(x_{T_{i+1}}^o)-u(x_o)&=u(x_{T_{i+1}}^o)-u(x_{T_{i+1}-T_i}^{o,i+1})+u(x_{T_{i+1}-T_i}^{o,i+1})-u(x_o)-(u(x_{T_{i}}^o)-u(x_o))\\
&\quad +(u(x_{T_{i}}^o)-u(x_o))\\
&=[((u(x_{T_{i+1}}^o)-u(x_{T_{i}}^o))-(u(x_{T_{i+1}-T_i}^{o,i+1})-u(x_o)))]+A^o_{i+1}\tau_{i+1}+\tau_{i+1}o(1)\\
&+\sum_{j=1}^iA^o_j\tau_j+(\tau_1+\dots+\tau_i)o(1).
\end{array}\end{equation}
What remains to be done is the estimate of the square bracket in the last line.
By using Lemma \ref{lema2}(i) we get
\begin{equation}\label{eqendest}
|((u(x_{T_{i+1}}^o)-u(x_{T_{i}}^o))-(u(x_{T_{i+1}-T_i}^{o,i+1})-u(x_o))|\leq C|x_{T_{i}}^o-x_o|(T_{i+1}-T_i).\end{equation}
We now need $|x_{T_{i}}^o-x_o|\leq \hat C(\tau_1+\dots+\tau_i)+o(1)$ to conclude and this does not seem to hold in general without further assumptions.
We can obtain it by Lemma \ref{lemestimate} with the extra assumption so that we finally get
\begin{equation}\label{eqaaestmanifold}
u(x_{T_{m}})-u(x_o)=A_o\tau+o(1)\sum_{i=1}^m\tau_i-\rho(x,\tau)=(A_o+\gamma(\tau))\tau-\rho(x,\tau),\end{equation}
with an appropriate continuous, vector valued function $\gamma(\tau)$ vanishing as $\tau\to0$.
In the case that $x_o\in\partial\mathcal T$ then with easy changes, if $\hat\rho(x,\tau)=\hat u(x^o_{T_m})-\hat u(x_{T_m})$, we obtain similarly
$$\hat u(x_{T_{m}})-\hat u(x_o)=(\hat A_o+\hat \gamma(\tau))\tau-\hat \rho(x,\tau).$$
We then complete the statement of the Theorem in both cases by using Lemma \ref{lempetrovter} in the Appendix similarly to the proof of Theorem \ref{teopoint}, since in the case $x_o\in\partial\mathcal T$ we need $u(x_{T_{m}})\geq u(x_o)$ and $u_{h+1}(x_{T_{m}})\geq u_{h+1}(x_o)$.
\end{proof}
\begin{rem}
Ideally we need the estimate (\ref{eqendest}) to have in the right hand side $C|u(x_{T_{i}}^o)-u(x_o)|(T_{i+1}-T_i)=(\tau_1+\dots+\tau_i)o(1)$, by the induction assumption, to be able to conclude as in Lemma \ref{lemestimate} without further assumptions. This does not seem to be reachable in general.\\
If $\bar k=2$ and $k_i=2$ then notice that the extra condition (A1) for the i-th group requires
$$0=(H_{f_1^{(i)}}\boxplus \dots\boxplus H_{f_{j(i)}^{(i)}})I(x_o)=H_{f_{1}+\dots+ f_{j(i)}}I(x_o)=
(f_1^{(i)}+\dots+f_{j(i)}^{(i)})(x_o),$$
i.e. the vector fields of the $i-$th group are balanced at $x_o$. This kind of assumption was made in the paper by the author \cite{so} to prove the corresponding result for second order sufficient conditions in the case of symmetric systems. Similarly one could make the extra assumption more transparent and direct on the vector fields also for higher values of the parameter.
\end{rem}
With a slight modification in the proof, we can in fact relax the extra assumption in the statement of Theorem \ref{teomanifold}, which is too restrictive in some examples, by using Lemma \ref{lema2}(ii) instead.
\begin{cor}\label{cormanifold}
Consider the nonlinear control system (\ref{eqsystem}) and $x_o\in\mathcal T$, where the target $\mathcal T$ is locally described in one of two ways in (\ref{eqvincolo}). Suppose that (B1) and (B2) are satisfied. 
In addition we require the following:
$B_R(x_o)\times A\ni (x,a)\mapsto Du(x)F(x,a)$ only depends on a restricted group of variables $x_{l_1},\dots,x_{l_p}$ and $a\in A$, and (\ref{eqassrestrict}) holds true, for all $x,y\in B_R(x_o)$ and $a\in A$.
Suppose also that the corresponding components of the vector field $F_{l_1},\dots,F_{l_p}$ also depend only on the same group of variables and the control as well, and
$$(H_{f_1^{(i)}}\boxplus \dots \boxplus H_{f_{j(i)}^{(i)}})^rI_j(x_o)=0,
$$
for all $i=1,\dots,m$, $j\in\{l_1,\dots,l_p\}$ and $1\leq r<k_i$, where $k_i$ comes from assumption (B1).
Then the target is STLA at $x_o$ and all other conclusions of Theorem \ref{teomanifold} hold true.
\end{cor}
\begin{proof} If $x_o\in\mathcal T\backslash\partial\mathcal T$, we proceed as in the proof of Theorem \ref{teomanifold} until we get to (\ref{eqaabbcc}). By using Lemma \ref{lema2} (ii) instead, we modify (\ref{eqendest}) according to the current additional assumption and obtain in the right hand side
$$C|(x_{T_{i}}^o-x_o)_{l_1,\dots,l_p}|(T_{i+1}-T_i)$$
instead, where the vector $x_{l_1,\dots,l_p}$ only contains the indicated $p$ coordinates. We now apply Lemma \ref{lema1} to the subsystem of the space coordinates $l_1,\dots,l_p$ and use it as in Lemma \ref{lemestimate} to finally achieve that
$$(x_{T_i}^o-x_o-\sum_{j=1}^iv^o_j\tau_j)_{l_1,\dots,l_p}=(\tau_1+\dots+\tau_i)o(1),
$$
where $v^o_j=(1/k_i!)(H_{f_1^{(i)}}\boxplus \dots \boxplus H_{f_{j(i)}^{(i)}})^{1/k_i}I_j(x_o)$ for $j\in\{l_1,\dots,l_p\}$. We conclude the induction step
$$u(x_{T_{i+1}}^o)-u(x_o)=\sum_{j=1}^{i+1}A^o_j\tau_j+(\tau_1+\dots+\tau_i)o(1)
$$
and again (\ref{eqaaestmanifold}) is satisfied. Similarly, we complete the argument if $x_o\in\partial\mathcal T$, concluding the proof by applying Lemma \ref{lempetrovter}(ii) in the Appendix and where the assumption (B2) is required.
\end{proof}
Another variation of the result is in the following.
\begin{cor}\label{cormanifoldbb} As in the previous Corollary assume (B1) and (B2) for the control system (\ref{eqsystem}) and $x_o\in\mathcal T$. Suppose in addition that for $j=0,1$ according to the fact that $x_o\in\mathcal T\backslash\partial \mathcal T$ or $x_o\in \partial\mathcal T$: the first $h+j$ columns of the jacobian $Du(x_o)$ (resp. $D(u,u_{h+1})(x_o)$) are not singular and that the last $(n-h-j)$ components of the system $(F_i)_{h+j+1,\dots,n}$, only depend on the corresponding coordinates $(x_i)_{h+j+1,\dots,n}$ and the control $a\in A$. Moreover
$$(H_{f_1^{(i)}}\boxplus \dots \boxplus H_{f_{j(i)}^{(i)}})^rI_l(x_o)=0,
$$
for all $i=1,\dots,m$, $l\in\{h+j+1,\dots,n\}$ and $1\leq r<k_i$, where $k_i$ comes from assumption (B1).
Then the target is STLA at $x_o$ and all other conclusions of Theorem \ref{teomanifold} hold true.
\end{cor}
\begin{proof}
We proceed as in the proof of Theorem \ref{teomanifold} until we get to (\ref{eqaabbcc}) and now estimate (\ref{eqendest}) as follows. Add $n-h-j$ components to $u=(u_1,\dots,u_{h+j})$ as $(\tilde u)_i(x)=x_i$, for $i=h+j+1,\dots,n$ in order to make $(u,\tilde u):\R^n\to\R^n$ locally invertible around $x_o$. If $L$ is a Lipschitz constant for the inverse function $(u,\tilde u)^{-1}$ then
$$|x_{T_{i}}^o-x_o|\leq L(|u(x_{T_{i}}^o)-u(x_o)|+|(x_{T_{i}}^o-x_o)_{h+j+1,\dots,n}|)\leq C(\tau_1+\dots+\tau_i),
$$
by the induction assumption, by Lemma \ref{lema1} applied to the subsystem of the space coordinates $h+j+1,\dots,n$ and Lemma \ref{lemestimate}. Eventually the left hand side in (\ref{eqendest}) estimates with $C(\tau_1+\dots+\tau_i)o(1)$ and we conclude the proof as before.
\end{proof}

\section {Examples}
In this section we show some examples illustrating our method.
\begin{exa} {(Bony type vector fields)} 
In $\R^3$, we consider a symmetric and convex controlled system $F(x,y,z,a_o,a_1)=f_1(x,y,z)a_1+f_2(x,y,z)a_2$, $f_1(x)= \;^t(z,z^2,0)$ and $f_2(x)= \;^t(0,0,1)$. We test our sufficient conditions for controllability with target 
${\mathcal T}=\{(x,y,z):x^2+y^2+z^2\leq \gep^2\}$
and $u(x,y,z)=(x^2+y^2+z^2-\gep^2)/2$. We check with $f(x)=(f_1(x)+f_2(x))/2$ and $g(x)=(f_1(x)-f_2(x))/2$.
We compute $H_f(u)=\nabla u\cdot f=\frac12(z-xz-yz^2)$ and $H_g(u)=\nabla u\cdot g=\frac12(-z-xz-yz^2)$. We check that on the boundary of the target they are both zero only on the circle
$$\left\{\begin{array}{ll}z=0,\\x^2+y^2+z^2=\gep^2.
\end{array}\right.$$
These are the points where we do not have first order decrease rate at the boundary of the target. Now we compute the coefficients for higher order controllability. We obtain
$$(H_f\boxplus H_g)^2u(x,y,0)=x/2,\quad  (H_g\boxplus H_f)^2u(x,y,0)=-x/2$$
therefore we have second order decrease rate at the points where $z=0,x\neq0$ by using a one switch trajectory. The only poins that remain to be checked are $(0,\pm\gep,0)$. We compute
$$(H_f\boxplus H_g)^3u(x,y,0)=y/2,\quad  (H_{-f}\boxplus H_{-g})^3u(0,y,0)=-y/2.$$
Thus we have third order controllability with a one switch trajectory following $-f$ first and $-g$ next at $(0,\gep,0)$ and following $f$ first and $g$ next at $(0,-\gep,0)$.
The analysis is then complete.\\
\end{exa}

\begin{exa}
In $\R^2$ with target $\mathcal T\{(x,y):x\leq 1\}$, we consider the affine system $F=f_0+af_1$, $a\in[-1,1]$ with $f_0(x,y)=\;^t(-y,x)$ and $f_1=2f_0$. Define $u(x,y)=x$. Then we obtain
$(f_0(x,y)+af_1(x,y))\cdot \nabla u(x,y)=-(1+2a)y$ so that a first order condition only fails at $(1,0)$. However we get
$$(H_{f_0+f_1}\boxplus H_{f_0-f_1})^2u(x,y)=-4x
$$
therefore a second order condition applies.
Notice that the Lie bracket $[f_1,f_0]\equiv 0$.
\end{exa}

\begin{exa}(From the book by Coron \cite{cor}.)
In $\R^2$ with target $\mathcal T=\{(x,y):x^2+y^2\leq r^2\}$ consider the affine vector field 
$F(x,y,a)=f_0(x,y)+af_1(x,y)$, $f_0(x,y)=\;^t(y^3,0),\;f_1(x,y)=\;^t(0,1)$. We compute
$$f(x,y,a)\cdot\nabla u(x,y)=y(xy^2+a),$$
which is negative on $\partial\mathcal T$ for some $|a|\leq1$ unless $y=0$ so that the points where a first order decrease rate condition fails are $(\pm r,0)$. Then we obtain
$$\begin{array}{c}
(H_{f_0+f_1}\boxplus H_{f_0-f_1})^2u(\pm r,0)=0,\;(H_{f_0+f_1}\boxplus H_{f_0-f_1})^3u(\pm r,0)=0,\\
(H_{f_0+f_1}\boxplus H_{f_0-f_1})^4u(x,y)=12x+204y^4,\; (H_{f_0-f_1}\boxplus H_{f_0+f_1})^4u(x,y)=-12x+204y^4,
\end{array}$$
therefore a fourth order condition holds at $(\pm r,0)$.\\
If instead we change the target to the origin then we get $(f_o\pm f_1)(0,0)=\;^t(0,\pm1)$ and also 
$$\begin{array}{cc}
(H_{f_0+f_1}\boxplus H_{f_0-f_1})I(0,0)=(0,0),\quad(H_{f_0+f_1}\boxplus H_{f_0-f_1})^2I(0,0)=(0,0),\\
(H_{f_0+f_1}\boxplus H_{f_0-f_1})^3I(0,0)=(0,0), (H_{f_0+f_1}\boxplus H_{f_0-f_1})^4I(0,0)=(12,0),\\
(H_{f_0-f_1}\boxplus H_{f_0+f_1})^4I(0,0)=(-12,0).
\end{array}$$
Therefore $\{(f_o\pm f_1)(0,0),(H_{f_0-f_1}\boxplus H_{f_0+f_1})^4I(0,0),(H_{f_0+f_1}\boxplus H_{f_0-f_1})^4I(0,0)\}$ is a positive basis of $\R^2$ and we have a 4th order condition at the origin.
\end{exa}

\begin{exa} (Nonlinear controlled damped and anharmonic oscillator) Notice that it is not an affine system.
We consider in $\R^2$ the controlled system where
$$F(x,y,a)=\left(\begin{array}{c}
y\\ax^3-x-a^2y\end{array}\right),\quad a\in[-1,1].
$$
We can compute, for target $\{(x,y):x^2+y^2\leq r^2\}$ and $u(x,y)=(x^2+y^2-r^2)/2$,
$$H_Fu(x,y)=ay(x^3-ay),$$
and we find out that at target points it is strictly negative for some $a$ except at the points $(\pm r,0)$. At these points we compute, for $f_1(x,y)=\;^t(y,-x-y)$, $f_2(x,y)=\;^t(0,x^3)$, as $f_1\pm f_2$ are available,
$$(H_{f_0+f_1}\boxplus H_{f_0-f_1})^2u(\pm r,0)=-2r^4$$
and therefore a second order condition is satisfied.
\end{exa}

\begin{exa}
In this example in $\R^3$ we consider the affine system with controlled vector field $F(x,y,z,a,b)=f_0(x,y,z)+af_1(x,y,z)+bf_2(x,y,z)$, $a,b\in[-1,1]$ where $f_0(x,y,z)=\;^t(-y,x,0)/12$, $f_1(x,y,z)=\;^t(xz,yz,0)$, $f_2(x,y,z)=\;^t(0,0,1)$ and target the cylinder $\{(x,y,z):x^2+y^2\leq r^2\}$ so that we may choose $u(x,y,z)=(x^2+y^2-r^2)/2$.
We observe that 
$$H_{f_0}u\equiv 0,\quad H^{(2)}_{f_0}u\equiv 0,\quad H_{f_1}u=(x^2+y^2)z,\quad H_{f_2}u=0.$$
In particular the points of the boundary of the target with $z\neq0$ satisfy a first order condition with $f_0+f_1$ or $f_0-f_1$.
We also compute 
$$(H_{f_o+f_2}\boxplus H_{f_o-f_1})^2u(x,y,z)=-2(x^2+y^2),$$
therefore a second order condition is satisfied at the points of the boundary of the target with $z=0$. 
\end{exa}

\begin{exa} (Reaching an axis in $\R^3$.)
We consider the affine system in $\R^3$ given by the controlled field $F(x,y,z,a)=f_o(x,y,z)+af_1(x,y,z)$ where $f_o(x,y,z)=\;^t(y,0,-z)$ and $f_1(x,y,z)=\;^t(0,1,0)$. The target is the $z-$axis $u(x,y,z)=(x,y,0)=0$. One checks that
$$H_{f_o\pm f_1}u(0,0,z)=\;^t(0,\pm1),\quad (H_{f_o\pm f_1}\boxplus H_{f_o\mp f_1})^2u(0,0,z)=\;^t(\pm1,0)
$$
and these four vectors are a positive basis of $\R^2$. In addition it is null at the origin
$$(H_{f_o\pm f_1}\boxplus H_{f_o\mp f_1})I(0,0,z)=2H_{f_o}I(0,0,z)=2f_o(0,0,z).$$
Therefore at the origin a sufficient (second order) condition is satisfied in view of Theorem \ref{teomanifold}. At  the other points of the axis we cannot apply the Theorem directly, but in this case the assumptions of Corollary \ref{cormanifold} hold true because $DuF=\;^t(y\;a)$ only depends on variable $y$ and the second component $F_2=a$ depends only on the same state variable.
We get that the $z-$axis is STLA.
\end{exa}

\begin{exa}{(Reaching a curve in $\R^3$.)}
Consider the affine system $F(x,y,z,a)=f_o(x,y,z)+af_1(x,y,z)$, $a\in[-1,1]$, where $f_o(x,y,z)=\;^t(0,z,0)$ and $f_1(x,y,z)=\;^t(0,0,1)$ and assume that we want to reach the curve $u(x,y,z)=(y-x^2,z)=(0,0)$ starting from a neighborhood of the origin. The origin is not STLA for the system because trajectories will never change the first coordinate. We easily compute at the origin
$$\begin{array}{ll}
A_o&=\left(H_{f_o+f_1}u(0)\;|\;H_{f_o-f_1}u(0)\;|\:(H_{f_o+f_1}\boxplus H_{f_o-f_1})^2u(0)\;|\;(H_{f_o-f_1}\boxplus H_{f_o+f_1})^2u(0)\right)\\
&=\left(\begin{array}{cccc}
0\;&0\;&2\;&-2\\
1&-1&0&0
\end{array}\right)\end{array}$$
and the columns of $A$ are a positive basis of $\R^2$. Therefore the curve is STLA at the origin by Corollary \ref{cormanifold} as before. A similar analysis can be performed at all points of the curve.
\end{exa}

\section{Appendix: a refinement of Petrov's Lemma}

The technical Lemma that we discuss in this section is a refinement of a result initially due to Petrov \cite{pe} and a of first extension in the paper by the author \cite{so}. It is based on the properties of a positive basis of $\R^k$. We will sketch the proof for the reader's convenience because it is one important step to compltete our arguments. 
We start recalling the concept of a positive basis of a vector space. 
\begin{defin} We say that a family of vectors $a_i\in \R^k,i=1,\dots m$ is a positive basis of $\R^k$ if for any $x\in\R^k$ there are nonnegative real numbers $\lambda_i\geq0$ such that $x=\sum_{i=1}^m\lambda_ia_i$. \end{defin}
We recall a few useful properties that a positive basis enjoys, see \cite{pe2} for more details:
\begin{itemize}
\item[(i)]{$\{a_i\}_{i=1,\dots,m}$ is a positive basis if and only if for any unit vector $x\in\R^k$, $|x|=1$, there is $i$ such that $a_i\cdot x<0$;}
\item[(ii)]{$\{a_i\}_{i=1,\dots,m}$ is a positive basis if and only if the cone $\cup_{\lambda\geq0}\lambda \;\hbox{co}\{a_1,\dots,a_m\}=\R^k$;}
\item[(iii)]{by compactness of the unit sphere and (i), one easily shows that: a positive basis $\{a_i\}$ remains so after a small perturbation, that is: there is $\gep>0$ sufficiently small so that if $\{b_i\}_{i=1,\dots,m}$ satisfies $|a_i-b_i|<\gep$, $i=1,\dots,m$, then $\{b_i\}_{i=1,\dots,m}$ is also a positive basis;}
\item[(iv)]{if $\{a_i\}$ is a positive basis, then there are strictly positive $\lambda_i>0$, which can therefore be taken as large as we wish, so that $\sum_{i=1}^m\lambda_ia_i=0$;}
\item[(v)]{a positive basis contains a basis of $\R^k$ as a vector space and conversely if we have a basis $v_i$, $i=1,\dots,k$ of $\R^k$, then adding to it the vector $v_{k+1}=-\sum_{i=1}^kv_i$ gives a positive basis.}
\end{itemize}

The following statement improves Petrov \cite{pe2} in part (i) and in parts (i-ii) slightly the author \cite{so}. Below we indicate $\R^m_\geq=\{\tau\in\R^m:\tau\geq0\}$ and $B_\sigma(x_o)=\{x:|x-x_o|\leq\sigma\}$.

\begin{lem}\label{lempetrovter}
Let $A_o\in \R^{k\times m}$ be a matrix whose columns are a positive basis of $\R^k$. 

\noindent
(i)
Let $\gamma:\R^n\times\R^m_\geq\to\R^{k\times m}$, $\rho:\R^n\times\R^m_\geq\to\R^k$ be continuous functions such that for some $\sigma\in(0,1)$,
\begin{equation}\label{eqlemass}
\lim_{\tau\to0,\tau\geq0}\sup_{x\in B_\sigma(x_o)}|\gamma(x,\tau)|=0;\quad
\lim_{x\to x_o}\sup_{\tau\in B_\sigma(0),\tau\geq0}|\rho(x,\tau)|=0.\end{equation}
Therefore there are $\delta,\delta'>0$ such that for any $x\in\R^n$, $|x-x_0|\leq \delta'$, we can find $\tau\in \R^m$, $\tau\geq0$, $|\tau|\leq\delta$ such that the following equation is satisfied
\begin{equation}\label{eqmainpetrov}
(A_o+\gamma(x,\tau))\tau=\rho(x,\tau),
\end{equation}
and moreover $|\tau|\leq K\sup_{|t|\leq\sigma,t\geq0}|\rho(x,t)|$, for some $K>0$ independent of $x$, $|x-x_o|\leq\delta'$.

\noindent
(ii)
Let $s\in \R^m$ be such that together with $A_o$ above they satisfy: 
the matrix
$$\tilde A_o=\left(\begin{array}{c}A_o\\^ts
\end{array}\right)$$
has rank $k+1$ and for all $p\in \R^k$ and $r\geq0$ there exists $\lambda=\;^t(\lambda_1,\dots,\lambda_m)\geq0$ such that
$$p=A_o\lambda,\quad r\leq s\cdot\lambda.$$
Let $\gamma:\R^n\times\R^m\to\R^{(k+1)\times m}$, $\rho:\R^n\times\R^m\to\R^{k+1}$ be continuous functions satisfying (\ref{eqlemass}).
Then there are $\delta,\delta'>0$ such that for any $x\in\R^n$, $|x-x_0|\leq \delta'$ we can find $\tau\in \R^m$, $\tau\geq0$, $|\tau|\leq\delta$ and $h(x)\geq0$ such that the following equation is satisfied
$$(\tilde A_o+\gamma(x,\tau))\tau=\rho(x,\tau)
+h(x)e_{k+1},$$
where $e_{k+1}=\;^t(0\dots0,1)\in\R^{m+1}$. Moreover $|\tau|\leq K\sup_{|t|\leq\sigma,t\geq0}|\rho(x,t)|$, for some $K>0$ independent of $x$, $|x-x_o|\leq\delta'$.
\end{lem}

\begin{proof} We outline the proof only in case (i) and refer to the references for more details. By the first limit in (\ref{eqlemass}) and rearranging the columns if necessary, 
we can assume that taking a sufficiently small $s>0$, for any $(x,\tau)\in B_\sigma(x_o)\times B_s(0)$ the columns of $A(x,\tau):=A_o+\gamma(x,\tau)$ are still a positive basis of $\R^k$ and we can split $A(x,\tau)=(A_1|A_2)$ where $A_1(x,\tau)$ is a $k\times k$ matrix such that $|\hbox{det}A_1(x,\tau)|\geq c_o>0$ for all $(x,\tau)\in B_\sigma(x_o)\times B_s(0)$.
Notice that $\gamma(x,0)=0$.
Let $M=\|A_1^{-1}\|_{L^\infty(B_\sigma(x_o)\times B_s(0))}$,  by property (iv) of the positive basis we can find $b_o\in\R^m$ such that $(b_o)_i\geq M+1$, $i=1,\dots,m$ and $A_ob_o=0$. 
Again by the first limit in (\ref{eqlemass}) we can find $0<\delta(\leq s)$ such that 
\begin{equation}\label{eqlembbb}
|A^{-1}_1(x,\tau)\gamma(x,\tau)b_o|\leq 1,\hbox{ for all } (x,\tau), \;|\tau|\leq\delta, \;|x-x_o|\leq\sigma,\;\tau\geq0.\end{equation}
Let $K=|b_o|+1+M$. 
By the second limit in (\ref{eqlemass}) choose $0<\delta'(\leq\sigma)$ such that 
$$|\rho(x,\tau)|\leq\frac\delta K\leq1,\hbox{ for all } (x,\tau), \;|\tau|\leq\sigma, \;|x-x_o|\leq\delta',\;\tau\geq0.$$

Now we can fix any $x,|x-x_o|\leq\delta'$, 
and define $b(x,\tau)=b_o+b_\#(x,\tau)$ where 
$$b_\#(x,\tau)=\;^t(-A^{-1}_1(x,\tau)
\gamma(x,\tau)b_o,0,\dots,0)\in\R^m$$
 so that $b_i(x,\tau)\geq M$ for all $\tau\geq0$, $|\tau|\leq\delta$, where we indicate $b(x,\tau)=(b_1,\dots,b_m)$. Note that
$$(A_o+\gamma(x,\tau))b(x,\tau)=\gamma(x,\tau)b_o+(A_o+\gamma(x,\tau))b_\#(x,\tau)=0,
$$
that $b$ is continuous by the inverse function theorem and
$b(x,0)=b_o$.

For any unit vector $X\in\R^k$ define $c(x,\tau,X)=\;^t(A^{-1}_1(x,\tau)X,0)\in\R^m$ and observe that $
|c(x,\tau,X)|\leq M$. 
Let $\beta(x,\tau,X)=b(x,\tau)
+c(x,\tau,X)$. 
Note that 
$$
(A_o+\gamma(x,\tau))\beta(x,\tau,X)
=(A_o+\gamma(x,\tau))c(x,\tau,X)
=
X,
$$
and that $\beta$ is continuous and such that $\beta\geq0$ since $|c|\leq M$, $b\geq M$.
We also obtain that $
|\beta(x,\tau,X)|\leq |b_o|+|b_\#(x,\tau)|+|c(x,\tau,X)|\leq K$ for all $|\tau|\leq\delta$, $|X|=1$, $\tau\geq0$. 

In order to solve (\ref{eqmainpetrov}), recall that $x$ has been fixed and define
$$\Phi(\tau)=\left\{\begin{array}{ll}
|\rho(x,\tau)|\beta(x,\tau,\frac{\rho(x,\tau)}{|\rho(x,\tau)|}),\quad& \hbox{ if }\rho(x,\tau)\neq0\\
0,&\hbox{ if }\rho(x,\tau)=0.
\end{array}\right.$$
Now observe that by construction $\Phi:B_\delta(0)\cap\{\tau:\tau\geq0\}\to B_\delta(0)\cap\{\tau:\tau\geq0\}$ is a continuous function and therefore it has a fixed point $\tau^*$ (since $B_\delta(0)\cap\{\tau:\tau\geq0\}$ is closed and convex). Such fixed point then satisfies
$$\begin{array}{l}(A_o+\gamma(x,\tau^*))\tau^*=|\rho(x,\tau^*)|(A_o+\gamma(x,\tau^*))\;\beta\left(x,\tau^*,\frac{\rho(x,\tau^*)}{|\rho(x,\tau^*)|}\right)=\rho(x,\tau^*)
\end{array}$$
if $\rho(x,\tau^*)\neq0$ otherwise the conclusion is obvious. Thus
$\tau=\tau*\geq0$ solves equation (\ref{eqmainpetrov}). Moreover the fixed point satisfies
$|\tau^*|\leq K|\rho(x,\tau^*)|$, $|\tau^*|\leq\delta$ and
therefore as we wanted
$$|\tau^*|\leq K\sup_{|\tau|\leq\sigma}|\rho(x,\tau)|.
$$

Case (ii) is an appropriate variation of case (i) and follows similarly to the proof of Proposition 3.2 in \cite{so}.
\end{proof}
\begin{rem}
The original result by Petrov \cite{pe} solves the $x-$independent equation
$(A+\gamma(\tau))\tau=v\in\R^k$.
The paper by the author solves
$(A+\gamma(x,\tau))\tau=H(x)$,
and adds case (ii) to the argument. Here we add the variable $\tau$ to the right hand side and indicate the needed assumptions in this case.\\
If the columns of $A_o$ are a positive basis of $\R^k$ and $b_o$ is such that $(b_o)_i>0$, for all $i$, and $A_ob_o=0$, then we may define eccentricity of the basis relative to $b_o$ as the ratio
$$E(b_o)=\frac{\min_i(b_o)_i}{\max(b_o)_i}.$$
This is stable with respect to small perturbations of the basis (and the choice of $b_o$). 
Looking at the proof, if $|\gamma(x,\tau)|\leq C|\tau|$, $|\rho(x,\tau)|\leq C|x-x_o|$, as it happens in our applications of the Lemma, then we easily check that the choices of $\delta,\delta'$ can be made of the order of $\hat CE(b_o)$, $\hat C$ a constant uniformly depending on the data, so they are also stable with respect to small perturbations of the basis.
Indeed to implement the choice of $\delta$ in (\ref{eqlembbb}) we find
$$|A^{-1}_1(x,\tau)\gamma(x,\tau)b_o|\leq MC|\tau|\sqrt{m}(M+1)/E(b_o)
$$
when we choose $b_o$ so that $M+1=\min_i(b_o)_i$.
\end{rem}


\end{document}